\documentclass[11pt,reqno]{amsart}

\usepackage{xcolor}
\usepackage[utf8]{inputenc}
\usepackage[T1]{fontenc}
\usepackage[english]{babel}

\usepackage{lmodern}
\usepackage{microtype}
\usepackage{amsmath,amssymb,amsthm}
\usepackage{amsthm, thmtools, thm-restate}
\usepackage{mathtools} 
\usepackage{bm}        
\usepackage{enumitem}

\usepackage{graphicx}  
\usepackage{caption}
\usepackage{subcaption}
\usepackage{tikz}      
\usepackage{mathrsfs}  
\usepackage{booktabs}  
\usepackage{geometry}
\usepackage{cancel}

\theoremstyle{plain}
\newtheorem{theorem}{Theorem}[section]      
\newtheorem{lemma}[theorem]{Lemma}          
\newtheorem{proposition}[theorem]{Proposition}
\newtheorem{corollary}[theorem]{Corollary}

\newtheorem*{theorem*}{Theorem}             

\theoremstyle{definition}
\newtheorem{definition}[theorem]{Definition}
\newtheorem{example}[theorem]{Example}

\newtheorem{remark}[theorem]{Remark}
\usepackage[hypertexnames=false]{hyperref} 


\numberwithin{equation}{section}

\newcommand{\R}{\mathbb{R}}
\newcommand{\N}{\mathbb{N}}

\usepackage{comment}

\title[Minkowski mean curvature equation: The Neumann Problem]{Positive Solutions for an Indefinite-Weight Minkowski Mean Curvature Neumann Problem: Multiplicity and Asymptotic Behaviour}

\author[R. Ziegele]{Ricardo Ziegele}
\address{Ricardo Ziegele \newline \indent 
Dipartimento di Matematica
\newline\indent
Universit\`a di Torino
\newline\indent
via Carlo Alberto 10, 10123 Torino, Italy}
\email{ricardoalfonso.ziegelealiaga@unito.it}

\begin{document}

\maketitle

\begin{abstract}
We study the Neumann boundary value problem
$$
-\operatorname{div}\left(\frac{\nabla u}{\sqrt{1-|\nabla u|^2}}\right)
= \lambda a(x)g(u) \quad \text{in } \Omega, \qquad
\frac{\nabla u}{\sqrt{1-|\nabla u|^2}}\cdot \mathbf n = 0 \quad \text{on } \partial\Omega,
$$
where $\Omega \subset \mathbb R^N$ is a bounded convex domain, $a\in L^\infty(\Omega)$ is an indefinite weight with $\int_\Omega a(x)  \, dx< 0$, and $g$ is a nonlinearity. Under suitable assumptions on $g$, we prove the existence of two positive solutions for sufficiently large $\lambda>0$, using Szulkin's theory for nonsmooth functionals: a global minimizer $u_\lambda^{(l)}$ with negative energy, and a mountain-pass critical point $u_\lambda^{(s)}$ with positive energy. 
Furthermore, we study the asymptotic behaviour as $\lambda\to+\infty$ of both families of solutions. For the ground states we prove that, up to a subsequence, $u_\lambda^{(l)}\to u_\infty$ uniformly, where $u_\infty$ satisfies $\|\nabla u_\infty\|_{L^\infty(\Omega)}=1$ and is characterized by being a maximizer of $\int_\Omega a(x) G(u) \, dx$ and a minimizer of the kinetic energy among all such maximizers. In the model case $g(u)=|u|^{p-2}u$, $p>2$, we also show that the mountain-pass energy level satisfies the upper bound $c_\lambda=O(\lambda^{-2/(p-2)})$, and that $u_\lambda^{(s)}\to0$ in $C(\overline\Omega)$.
\end{abstract}

\tableofcontents

\section{Introduction}

The problem of interest is

\begin{equation}
\label{eq-general}
    \begin{cases}
        \begin{aligned}
            &-\operatorname{div}\left(\frac{\nabla u}{\sqrt{1-|\nabla u|^2}} \right) = \lambda a(x) g(u) \quad &&\text{in } \Omega, \\
            &\frac{\nabla u}{\sqrt{1-|\nabla u|^2}} \cdot \mathbf{ n} = 0 &&\text{on } \partial\Omega, 
        \end{aligned}
    \end{cases}
\end{equation}
where $\Omega \subset \mathbb{R}^N$ is a bounded, convex domain, $N \geq 2$, $\partial \Omega \in C^2$, $\mathbf{ n}$ is the exterior normal vector,  $a : \Omega \to \R$ is an indefinite $L^\infty(\Omega)$ weight,  and $g: \R \to \R$ is a nonlinear term. 

The differential operator appearing in \eqref{eq-general} arises naturally in areas such as differential geometry and general relativity, since it can be meant as a prescribed mean curvature operator for $N$-dimensional spacelike Cartesian hypersurfaces in the $(N+1)$-dimensional Lorentz-Minkowski space. It also appears in models coming from mathematical physics, such as the Born--Infeld theory in nonlinear electrodynamics. As a general, non-exhaustive overview on the matter, we refer to the seminal works \cite{BaSi82,ChYa76,Ge83}, as well as to some recent contributions \cite{BoCoFo19,BoIa19,ByIkMaMa24,MaMa25,Ma26} for general results, especially in the direction of regularity issues for equations involving this operator.

From a nonlinear analysis perspective, problems involving this differential operator have been object of great interest during the last few decades, specifically with a Dirichlet boundary condition, for which the regularity results established in \cite{BaSi82} ensure the non-degeneracy of the operator. We refer to \cite{Az14,Az16,BeJeMa14,BeJeTo13,CoObOmRi13,CoObOmRi13b,Da16,Da17,DaWa17} and the references therein for the study of existence of solutions of the Dirichlet problem, using topological methods, variational methods, and bifurcation theory.

In what concerns equations such as \eqref{eq-general} with a Neumann boundary condition, the literature is considerably more limited. To the best of our knowledge, existing results are confined to the radial setting, where the problem reduces to a second-order ODE and can be treated by means of shooting methods or phase-plane techniques (see, for instance, \cite{BoFe20,BoFe202,BoCoNo20,BoFeZa16,BoFeZa23}). The general PDE case, though, has remained essentially untouched.

The main obstruction lies in the lack of regularity results guaranteeing that $|\nabla u|$ stays strictly below $1$ up to the boundary. More specifically, the natural variational framework for this problem is defined in the closed convex set
$$K:= \{ u \in W^{1,\infty}(\Omega) \, : \, \| \nabla u \|_{L^{\infty}(\Omega)} \leq 1 \},$$
whose elements are usually referred to in the literature as
\emph{weakly spacelike}. In particular, critical points of the associated
functional might, in principle, lie on the boundary of $K$, where the
operator degenerates, and therefore need not be strictly spacelike weak
solutions in the sense of Definition~\ref{def:solution}, which are the
solutions of interest in this work.

In the Dirichlet setting, this difficulty is overcome by the regularity results in \cite{BaSi82}, since the boundary datum itself provides a control mechanism on the gradient near $\partial\Omega$. For the Neumann problem, however, an analogous regularity result was not available for general domains. In the radial setting, the symmetry reduces the equation to a one-dimensional ODE, allowing one to exploit the specific structure of the resulting initial value problem with $(u'(0),u'(R))=(0,0)$ and obtain more precise information on the gradient. This approach, however, is inherently tied to radial symmetry and does not provide a corresponding instrument for general domains.

This gap has recently been closed in \cite{Ma26}, where the first gradient regularity result for the Neumann problem associated with the Minkowski mean curvature operator is established for \textbf{convex} domains $\Omega$. In particular, bounded weak solutions of the Neumann problem are shown to be of class $W^{2,2}(\Omega)\cap C(\overline{\Omega})$ up to the boundary, with $|\nabla u|$ uniformly bounded away from $1$ (cf. Lemma~\ref{lemma:existence}). This result opens the door, for the first time, to a variational treatment of \eqref{eq-general} with Neumann boundary conditions in the general PDE setting. Following the variational approach established in \cite{BeJeMa14} by means of Szulkin's non-smooth critical theory, it allows us to guarantee the correspondence between solutions of \eqref{eq-general} and critical points of the associated functional (cf. Theorem~\ref{thm:critical-point-equivalence}).

Motivated by this, in this paper we study \eqref{eq-general} by non-smooth variational methods, in the presence of an indefinite weight $a(x)$ and a nonlinearity $g$ with suitable growth conditions. Elliptic problems involving indefinite (sign-changing) weights have a long history in nonlinear analysis; hence, a considerable literature has been devoted to existence, multiplicity, and qualitative properties of positive solutions for semilinear elliptic equations with indefinite nonlinearities; see, among others, \cite{HeKa80,BCN95,AlTa96,AmLo98} and the references therein.

Problems with indefinite weights for the Minkowski mean curvature operator have been object of study in recent years, mainly in one-dimensional or radial settings; see, for instance, \cite{BoFe20,BoFe202,BoFeZa16,BoFeZa23}. In particular, the present work may be viewed as a PDE counterpart, on convex domains, of the results in the radial setting of \cite{BoFe202,BoFe20}.

For the purpose of presenting our contributions, we restrict our attention below to the following model problem, although most of our results apply to a much broader class of nonlinear terms.
\begin{equation}\label{eq-model}
\begin{cases}
\displaystyle -\text{div}\left( \frac{\nabla u}{\sqrt{ 1 - |\nabla u|^2}} \right) = \lambda a(x)|u|^{p-2}u \qquad &\mbox{in }{\Omega},\\
\frac{\nabla u}{\sqrt{ 1 - |\nabla u|^2}} \cdot \mathbf{n}=0 &\mbox{on }{\partial\Omega},
\end{cases}
\end{equation}
where $\Omega$ is a convex bounded domain, $\lambda >0$, $p>2$ and $a\in L^\infty(\Omega)$ is sign-changing, satisfies $\int_\Omega a(x)\,dx<0$ and $\operatorname{int}\{a>0\}\neq\emptyset$.

\begin{theorem}\label{thm:intro}
Let $p>2$ and assume that $a\in L^\infty(\Omega)$, $\int_\Omega a(x)\,dx<0$, and $\operatorname{int}\{a>0\}\neq\emptyset$. The following holds:
\begin{itemize}
\item \textbf{Existence and multiplicity.} There exists $\Lambda>0$ such that, for every $\lambda>\Lambda$, problem \eqref{eq-model} has at least two positive solutions, $u_\lambda^{(s)}$ and $u_{\lambda}^{(l)}$.
\item 
\textbf{Asymptotics.} As $\lambda\to+\infty$,
$$u_\lambda^{(s)}\to0\quad\text{uniformly in }C(\overline\Omega),$$
and, up to a subsequence,
$$u_\lambda^{(l)}\to u_\infty\quad\text{uniformly in }C(\overline\Omega),$$
where $u_\infty$ is a nonnegative function in $K_a$ satisfying $\|\nabla u_\infty\|_{L^\infty(\Omega)}=1$. Moreover, $u_\infty$ is a solution of 
$$\max_{u \in K}\int_\Omega a(x)|u|^p\,dx$$ and minimizes $\int_\Omega 1-\sqrt{1-|\nabla u|^2} \, dx$ among all such maximizers. On every open subset $U\subset\Omega$ such that $\|\nabla u_\infty\|_{L^\infty(U)}<1$, one has $a(x)u_\infty^{p-1}=0$ a.e. in $U$; in particular, if $a\neq0$ a.e. in $U$, then $u_\infty\equiv0$ in $U$.
\end{itemize}
\end{theorem}

Here, the superscripts $(s)$ and $(l)$ stand for \emph{small} and \emph{large}, respectively, referring to the different asymptotic behaviour of the two families of solutions as $\lambda\to+\infty$. 

We remark that the convexity of $\Omega$ is required only to apply the Neumann solvability and regularity results in \cite{Ma26}, which are currently available only on such domains. The remaining arguments do not require convexity itself: the bounds used in this work could be easily obtained on bounded, connected domains with sufficiently regular boundary, with constants depending on the domain. Consequently, analogous regularity results on such domains would allow the results of this paper to extend to that setting.

Let us comment on our results.

The existence and multiplicity part of Theorem~\ref{thm:intro} extends to the general, non-radial PDE setting with Neumann boundary conditions the non-smooth variational framework developed in \cite{BeJeMa14} for the Dirichlet problem, here in the presence of a nonlinearity with an indefinite weight. More precisely, once $\lambda$ exceeds a threshold value $\Lambda$, we obtain a positive global minimizer $u_\lambda^{(l)}$ of negative energy, which we will be referring to, as usual in the literature, as a \emph{ground state}, together with a positive mountain-pass critical point $u_\lambda^{(s)}$ of positive energy (see Theorems~\ref{thm:GS} and~\ref{thm:MP}). The positivity of the latter requires an additional argument in the non-smooth setting. To this end, we establish an abstract retraction principle for mountain-pass critical points of Szulkin-type functionals, in the spirit of \cite[Theorem 10]{BCN95} for smooth functionals. This principle shows that the mountain-pass level can be attained within the range of a suitable energy-decreasing nonexpansive retraction $p$ (see Proposition~\ref{prop:abstract-positivity}). Applying it to the modulus map $p(u)=|u|$ yields a nonnegative mountain-pass critical point, whose strict positivity then follows from standard positivity arguments for uniformly elliptic operators.

A further part of the paper concerns the asymptotic behaviour as $\lambda\to+\infty$. We first show that the ground states $u_\lambda^{(l)}$ converge uniformly, up to a subsequence, to a limiting profile $u_\infty$, characterized as a maximizer of the limiting functional over $K$ (see Proposition~\ref{prop:asGS}), by adapting the arguments developed in \cite{BoCoZi26}. For the model nonlinearity $g(u)=|u|^{p-2}u$, $p>2$, the mountain-pass solutions $u_\lambda^{(s)}$ converge uniformly to $0$ (see Proposition~\ref{prop:MP-asymptotics}). A key ingredient in this analysis is the gradient constraint $\|\nabla u\|_{L^\infty(\Omega)}\leq1$, which plays a central role in the compactness argument underlying these limits (see Lemma~\ref{lemma:compactness}).

The main difference between the limiting profile $u_\infty$
obtained here and that arising in the Dirichlet case is that, in the latter,
the boundary condition allows for an explicit characterization of the limiting
profile, under a monotonicity assumption on the nonlinearity, in terms
of the function $\operatorname{dist}(\cdot,\partial\Omega)$
(cf.~\cite[Corollary 2.19]{BoCoZi26}). In the present Neumann setting, the indefinite weight and the boundary condition do not allow us to follow the same argument. We are therefore led to study directly the variational problem that characterizes $u_\infty$.

The limiting profile exhibits a further geometric property: the gradient constraint is saturated, namely,
$$
\|\nabla u_\infty\|_{L^\infty(\Omega)}=1,
$$
meaning that it lies on the boundary of the admissible set of solutions, $K$. 
This naturally raises the question of where the constraint remains active. On every open set $U$ where $\|\nabla u_\infty\|_{L^\infty(U)}<1$, we show that $a(x)g(u_\infty(x))=0$ a.e. Thus, wherever the weight does not vanish in such a set, the profile must be zero (see Proposition~\ref{prop:inactive-region}).

Where the weight vanishes, a different mechanism comes into play. Admissible changes to the profile within these regions do not affect the limiting functional, so what selects its shape? We show that $u_\infty$ minimizes the kinetic energy $\Psi(u) = \int_\Omega 1-\sqrt{1-|\nabla u|^2} \, dx$ among all maximizers of the limiting functional (see Lemma~\ref{prop:selection}). This helps explain how constant plateaux may arise: on a connected Lipschitz domain $D\subset\subset \Omega$ where $a=0$ a.e., a constant boundary trace forces the profile to remain constant throughout $D$ (see Corollary~\ref{cor:limitprofile-plateaus}). Such domains need not exist, however, since the weight may fail to vanish a.e. on any nonempty open set. Conversely, any positive plateau on an open set requires the weight to vanish there a.e. These results alone do not determine whether plateaux occur or which heights they take; in particular, they do not require different plateaux to share a common height.

A further question concerns the size and geometry of the contact set
$$
\Lambda_0(u_\infty):=\{x\in\Omega:|\nabla u_\infty(x)|=1\},
$$
understood up to a null set, and how these depend on the weight. Even if no open set has a uniformly inactive constraint, full saturation almost everywhere does not follow. Understanding the contact set, as well as finding conditions on the weight that force plateaux to occur, remains a natural direction for further study.

The paper is organized as follows. In Section~\ref{sec2} we set up the variational framework associated with \eqref{eq-general}: we introduce the relevant functional $I_\lambda$ on the space $K$ of $1$-Lipschitz functions, establish the equivalence between critical points of $I_\lambda$, in the sense of Szulkin, and weak solutions of \eqref{eq-general} (cf. Theorem~\ref{thm:critical-point-equivalence}), and show that $I_\lambda$ attains its infimum. In Section~\ref{sec3} we use this variational framework to prove, for $\lambda$ sufficiently large, the existence of two distinct non-trivial solutions: a ground-state solution and a mountain-pass solution, with negative and positive energy levels, respectively (cf. Theorems~\ref{thm:GS} and~\ref{thm:MP}). In Section~\ref{sec4} we study their asymptotic behaviour as $\lambda\to+\infty$ (cf. Propositions~\ref{prop:asGS} and~\ref{prop:MP-asymptotics}) and the properties of the resulting limiting profile. Finally, in Appendix~\ref{app:mountain-pass} we prove the abstract retraction principle used to obtain the positive mountain-pass solution (cf. Proposition~\ref{prop:abstract-positivity}).

\section*{Acknowledgments}
The author is grateful to A. Boscaggin for suggesting the problem and, together with F. Colasuonno and D. Céspedes, for carefully reading the manuscript, and to L. Maniscalco for helpful discussions.

\section{The variational framework}\label{sec2}
Throughout this section, we introduce the variational and analytical setting in which \eqref{eq-general} will be studied, and we establish the compactness and existence properties needed to apply critical point theory for nonsmooth (Szulkin-type) functionals.

We consider $\lambda>0$ and impose the following
assumptions on $a$ and $g$:
\begin{itemize}
    \item[$(a_*)$] $a \in L^\infty(\Omega)$ and $\int_\Omega a(x) dx < 0;$  
    \item[$(a_+)$] The interior of the set $\{x \in \Omega \, \colon \, a(x) > 0 \}$ is not empty;
    \item[$(g_1)$] $g \in C^1(\R)$ and $g(s) s > 0$ for every $s \not = 0$;
    \item[$(g_2)$] There exist $p>1$ and $c>0$ such that $\lim_{|s|\to+\infty}
\frac{g(s)}{|s|^{p-2}s}=c$;

    \item[$(g_3)$] $ \lim_{s \to 0} \frac{g^2(s)}{G(s)} = 0$, where $G(s) := \int_0^s g(t) \, dt$.
\end{itemize}
Clearly the model case $g(u) = |u|^{p-2}u $ with $p>2$ satisfies ($g_1$)-($g_2$)-$(g_3)$.
Let us now comment about these hypotheses.
The hypothesis $(a_*)$ has a two-fold use; on the one hand, the negative-mean property is needed to ensure the coercivity of the associated functional (cf. Lemma \ref{boundness}), and on the other hand, the boundedness of $a(x)$ is necessary to establish the equivalence between critical points and solutions (cf. Theorem \ref{thm:critical-point-equivalence}). The same negative-mean hypothesis has appeared in the past in the radial setting, which is a necessary condition to ensure the existence of solutions when $g$ is strictly increasing (see, for instance, \cite{BoFe20}). The hypothesis $(a_+)$ is a technical condition to ensure the existence of non-trivial ground-state solutions (cf. Theorem \ref{thm:GS}). Hypothesis $(g_1)$ is a sign condition, ensuring that $G$ is positive and increasing for $s \geq 0$; hypotheses $(g_2)$ and $(g_3)$ impose, respectively, growth conditions at infinity and at the origin. Finally, we remark that, under $(g_1)$, condition $(g_3)$ is equivalent to 
$$g'(0)=0.$$

\subsection{Preliminaries}

Following the approach in \cite{BeJeMa14}, adapted to the results obtained in \cite{Ma26}, in this paper we restrict our analysis to the so-called \emph{weak strictly spacelike solutions} (we refer to \cite{BaSi82} for this terminology). We start by making precise the notion of solution of \eqref{eq-general} that we work with throughout the paper.

\begin{definition}\label{def:solution}
A function $u$ is a \emph{(weak) solution} of \eqref{eq-general} if $u \in W^{1,\infty}(\Omega)$,  $\|\nabla u\|_{L^\infty(\Omega)} < 1$, and
$$
\int_\Omega \frac{\nabla u \cdot \nabla \phi}{\sqrt{1-|\nabla u|^2}} \, dx = \lambda \int_\Omega a(x) g(u) \phi \, dx \qquad \text{for every } \phi \in C^1(\overline\Omega).
$$
\end{definition}

As we will see below (cf.\ Remark~\ref{remark:density}), since $\|\nabla u\|_{L^\infty(\Omega)}<1$ for a solution $u$ in the sense of Definition~\ref{def:solution}, both sides of the identity above extend continuously to test functions $\phi \in W^{1,2}(\Omega)$, by density of $C^1(\overline\Omega)$ in $W^{1,2}(\Omega)$; we will freely use this extended formulation whenever convenient.

To study \eqref{eq-general} variationally, following the approach of \cite{BeJeMa14}, we work within the closed convex set of $1$-Lipschitz functions
$$K := \{ u \in W^{1,\infty}(\Omega) \, \colon \, \|\nabla u \|_{L^\infty(\Omega)} \leq 1 \} \subset C(\overline\Omega),$$
and we introduce the nonsmooth functional $I_\lambda \colon C(\overline\Omega) \to (-\infty, +\infty]$,
\begin{equation}
    \label{functional}
    I_\lambda(u) = \Psi(u) + \Phi_\lambda(u),
\end{equation}
where
$$\Psi(u) := 
\begin{cases}
        \displaystyle \int_\Omega \left( 1- \sqrt{1-|\nabla u|^2} \right) dx \quad &\text{if } u \in K, \smallskip \\
        \displaystyle 
        + \infty &\text{if } u \in C(\overline{\Omega}) \setminus K,
\end{cases}
    $$ 
and 
$$\Phi_\lambda(u) := - \lambda \int_\Omega a(x)G(u) dx,$$
for all $u\in C(\overline{\Omega})$. As in \cite{BeJeMa14,Sz86}, the functional $I_\lambda=\Psi+\Phi$ has the structure required by Szulkin's critical point theory, namely: 
$$
\begin{aligned}  
&I_{\lambda} = \Psi + \Phi_{\lambda} \text{ is defined over a real Banach space (namely $(C(\overline{\Omega}),\|\cdot\|_{L^\infty(\Omega)})$)},\\   
&\Psi : C(\overline{\Omega}) \to (-\infty, +\infty] \text{ is convex, proper (i.e., $\Psi \not\equiv +\infty$), and lower semicontinuous},\\
&\Phi_{\lambda} \in C^1(C(\overline{\Omega});\mathbb{R}), 
\end{aligned} 
$$
see \cite[Lemma 2.4]{BeJeMa14} for the lower semi-continuity and convexity of $\Psi$. With the functional $I_\lambda$ at hand, we can now make precise the notions of critical point and Palais--Smale sequence that we will work with.

\begin{definition}\label{def:crit-pt}
A function $u \in K$ is called a \emph{critical point} of $I_\lambda$ if it is a solution of the variational inequality
$$
\Psi(v) - \Psi(u) + \Phi_\lambda'(u)[v-u]  \geq 0 \, \text{ for all } v \in K,
$$
namely, for all $v \in K$
$$
\int_\Omega \big( 1- \sqrt{1-|\nabla v|^2} \big)dx - \int_\Omega \big( 1- \sqrt{1-|\nabla u|^2} \big)dx- \lambda \int_\Omega a(x)g(u)(v-u)dx \geq 0. 
$$
\end{definition}
\begin{definition}\label{def:PS}
$I_\lambda$ satisfies the \emph{Palais-Smale condition} (shortened $(PS)$-condition) if any sequence $\{u_n\} \subset K$ with $I(u_n)\to c\in\mathbb R$ satisfying 
\begin{equation}
\label{ineq:PS}
\Psi(v) - \Psi(u_n) + \Phi'_\lambda(u_n)[v-u_n] \geq -\varepsilon_n\|v-u_n\|_{L^\infty(\Omega)} \, \text{ for all } v \in K, 
\end{equation}
where $\varepsilon_n \to 0^+$, has a convergent subsequence. 
\end{definition}

\begin{remark}\label{rmk:K-noncompact}
Unlike the Dirichlet setting of \cite{BeJeMa14}, where the boundary datum $u|_{\partial \Omega} = 0$ combined with the $1$-Lipschitz bound gives an a-priori
$L^{\infty}(\Omega)$-estimate (namely, $\|u\|_{L^{\infty}(\Omega)} \leq \tfrac 12 \operatorname{diam}(\Omega)$), and hence yields the compactness of the admissible set
itself, the space $K$ introduced above contains every constant function and
is therefore unbounded in $C(\overline{\Omega})$. Consequently $K$ is
not compact, and boundedness of Palais-Smale sequences cannot be inherited from the geometry of
$K$. Instead, it must be proved using a combination of the negative-mean property of the weight  $(a_\ast)$, and the growth of $g$ at infinity $(g_2)$, which together serve to ensure a coercivity-type property of the functional $I_\lambda$, cf. Lemma~\ref{boundness} below. This is the first point at which the Neumann boundary condition, which
imposes no a-priori constraint on the average of $u$, makes the
variational analysis different from the Dirichlet case; it is
also the reason a dedicated compactness argument, of a different nature,
will be required in Section~\ref{sec4} to carry out the asymptotic analysis
as $\lambda \to +\infty$ (see
Lemma~\ref{lemma:compactness}).
\end{remark}

We now show that $I_\lambda$ enjoys a coercivity-type property on $K$, which will be the key ingredient in proving that $I_\lambda$ does satisfy the $(PS)$-condition.

\begin{lemma} \label{boundness}
Assume $(a_*)$, $(g_1)$ and $(g_2)$. Let $\{u_n\}\subset K$ be a sequence such that $$\sup_n |I_\lambda(u_n)|<+\infty.$$ Then $\|u_n\|_{L^\infty(\Omega)}$ is bounded.
\end{lemma}
\begin{proof}
Assume by contradiction that there exists a sequence $\{u_n\} \subset K$ such that $|I_\lambda(u_n)|$ is bounded and, up to a subsequence,
$$
M_n:=\|u_n\|_{L^{\infty}(\Omega)} \rightarrow +\infty
\quad\text{as } n\to+\infty.
$$
Let $x_n\in\overline{\Omega}$ be such that
\[
|u_n(x_n)|=M_n.
\]
Up to a subsequence, either $u_n(x_n)=M_n$ for every $n$, or
$u_n(x_n)=-M_n$ for every $n$. Assume first that
$u_n(x_n)=M_n$ and set $D:=\operatorname{diam}(\Omega)$. Since every function in $K$ is $1$-Lipschitz, for every $x\in\overline{\Omega}$,
\begin{equation}\label{ineq:max}
u_n(x_n)\ge u_n(x)\ge u_n(x_n)-|x-x_n|
\ge u_n(x_n)-D.
\end{equation}
It follows from \eqref{ineq:max} that
$$
u_n(x)\ge M_n-D>0,
$$
for all $x\in\overline{\Omega}$ and $n$ sufficiently large, since $M_n\to+\infty$. The case where $u_n(x_n)=-M_n$ is completely symmetric, thanks to $(g_1)$.

\medskip

Next, notice that $(g_1)$ implies that
$$
G(s)=\int_0^s g(t)\,dt,
$$
is increasing for $s > 0$ (resp. decreasing for $s\leq 0)$, and non-negative for every $s \in \R$. Therefore, by \eqref{ineq:max},
\begin{equation}\label{ineq:G}
G(M_n)\ge G(u_n(x))\ge G(M_n-D),
\qquad
\forall\,x\in\overline{\Omega}.
\end{equation}
Since $\Psi(u)\ge0$ for every $u\in K$, we have
$$
I_\lambda(u)\ge
-\lambda\int_\Omega a(x)G(u)\,dx
=
-\lambda\left(
\int_{\{a\ge0\}}a(x)G(u)\,dx
+
\int_{\{a<0\}}a(x)G(u)\,dx
\right).
$$
Using \eqref{ineq:G}, we deduce that
\begin{equation}\label{ineq:I1}
I_\lambda(u_n)\ge
-\lambda\left(
G(M_n)\int_{\{a\ge0\}}a(x)\,dx
+
G(M_n-D)\int_{\{a<0\}}a(x)\,dx
\right).
\end{equation}
Now set
\begin{equation} \label{eq:notation}
A^+:=\int_{\{a\ge0\}}a(x)\,dx,
\qquad
A^-:=-\int_{\{a<0\}}a(x)\,dx.
\end{equation}
By assumption $(a_*)$,
$$
0\le A^+<A^-.
$$
Moreover, $(g_2)$ yields
$$
G(s)\sim \frac{c}{p}|s|^p
\qquad\text{as }|s|\to+\infty.
$$
Hence,
$$
G(M_n-D)\to+\infty,
$$
and
$$
G(M_n)-G(M_n-D)
=
\int_{M_n-D}^{M_n}g(t)\,dt
=
O(M_n^{p-1})
=
o\bigl(G(M_n-D)\bigr).
$$
Combining these estimates with \eqref{ineq:I1}, we obtain
\begin{equation} \label{eq:coerc}
\begin{aligned} 
I_\lambda(u_n)
&\ge
\lambda\Bigl[
(A^- - A^+)G(M_n-D)
-
A^+\bigl(G(M_n)-G(M_n-D)\bigr)
\Bigr] \\
&=
\lambda G(M_n-D)
\left[
A^- - A^+ + o(1)
\right]
\longrightarrow+\infty,
\end{aligned}
\end{equation}
since $A^->A^+$, thus reaching a contradiction and finishing the proof.
\end{proof}

This a priori bound on Palais-Smale sequences is precisely what is needed, together with the closeness of $K$ in $C(\overline\Omega)$, to obtain the following compactness property of $I_\lambda$.

\begin{corollary}\label{cor:PS}
    $I_\lambda(u)$ defined as in \eqref{functional} satisfies the (PS)-condition.
\end{corollary}
\begin{proof}
   Lemma~\ref{boundness} \ yields that every Palais-Smale sequence $\{u_n\}\subset K$ is bounded in $L^\infty(\Omega)$. Since every function in $K$ is $1$-Lipschitz, the sequence is equi-continuous. Hence, by the Arzelà-Ascoli theorem, up to a subsequence,
$$
u_n\to u
\quad\text{uniformly in } C(\overline\Omega).
$$
Since $K$ is closed in $C(\overline\Omega)$, we have $u\in K$. Finally, passing to the limit in inequality \eqref{ineq:PS}, using the continuity of $\Phi'$ and the lower semicontinuity of $\Psi$, we conclude that every Palais-Smale sequence admits a convergent subsequence. Therefore, $I_\lambda$ satisfies the $(PS)$-condition.
\end{proof}

In order to relate critical points of $I_\lambda$ to weak solutions of \eqref{eq-general}, we will need a solvability result for the underlying Neumann boundary condition problem for the mean curvature operator. To this end, we recall the following existence and regularity result, due to Maniscalco.

\begin{lemma}[{\cite[Theorem 1.1]{Ma26}}]
\label{lemma:existence}
If $\rho \in L^{\infty}(\Omega)$ is such that $\int_\Omega \rho(x) \, dx = 0,$ then the problem
\begin{equation}\label{eq-rho} -\operatorname{div} \left(\frac{\nabla u}{\sqrt{1-|\nabla u|^2}} \right) = \rho \quad \text{in } \Omega, \quad \frac{\nabla u}{\sqrt{1-|\nabla u|^2}}\cdot \mathbf n=0\ \text{ on }\partial\Omega,\end{equation}
has a weak solution $u_\rho \in W^{2,2}(\Omega) \cap C(\overline{\Omega})$, and it is unique up to translations. Moreover, it satisfies $\|\nabla u_\rho\|_{L^{\infty}(\Omega)} < 1- \theta$, where $\theta = \theta(N,\Omega, \|\rho\|_{L^{\infty}(\Omega)})$.
\end{lemma}

\begin{remark}\label{remark:density}
The notion of weak solution of \eqref{eq-rho} used in \cite{Ma26} is exactly the one of Definition~\ref{def:solution}, with $\rho$ in place of $\lambda a(x) g(u)$: namely, $u \in W^{2,2}(\Omega)\cap C(\overline{\Omega})$ such that
$$
\int_\Omega \frac{\nabla u\cdot\nabla\phi}{\sqrt{1-|\nabla u|^2}}\,dx
=
\int_\Omega \rho\phi\,dx
\qquad\text{for every }\phi\in C^1(\overline{\Omega}).
$$
In our work, though, we will use the equivalent formulation with test functions in $W^{1,2}(\Omega)$. Indeed, since under our hypothesis $C^1(\overline{\Omega})$ is dense in $W^{1,2}(\Omega)$ and
$$
\frac{\nabla u}{\sqrt{1-|\nabla u|^2}}\in L^\infty(\Omega)
$$
whenever $\|\nabla u\|_{L^\infty(\Omega)}<1$, both sides of the above identity define continuous linear functionals on $W^{1,2}(\Omega)$. Hence, the identity extends by density to every $\phi\in W^{1,2}(\Omega)$.
\end{remark}

\begin{lemma} \label{lemma:crit}
    If $\rho \in L^{\infty}(\Omega)$ is of zero-mean, then $u_\rho$ from Lemma \ref{lemma:existence} is a solution in $K$ of the variational inequality
    \begin{equation}
    \label{ineq:var1}
    \int_{\Omega} [\sqrt{1-|\nabla u|^2} - \sqrt{1-|\nabla v|^2} - \rho (v-u) ] \, dx \geq 0 \quad \text{for all } v\in K.  \end{equation}
    Moreover, $u_\rho$ is the unique solution in $K$ of \eqref{ineq:var1}, up to an additive constant.
\end{lemma}
\begin{proof}

    By $u_\rho$ being a weak solution of \eqref{eq-rho}, we have that
    $$\int_\Omega \frac{\nabla u_\rho \cdot \nabla \phi}{\sqrt{1-|\nabla u_\rho|^2}} \,dx = \int_\Omega \rho \phi \,dx \quad \text{for every } \phi \in C^1(\overline{\Omega}).$$
    Taking $\phi = v-u_\rho \in W^{1,2}(\Omega)$ (cf. Remark \ref{remark:density}), for some arbitrary $v \in K$, we obtain
    $$\int_\Omega \left[  \frac{ \nabla u_\rho \cdot \nabla(v-u_\rho) }{\sqrt{1-|\nabla u_\rho|^2}} - \rho(v-u_\rho)\right] dx = 0, $$
    which, using the concavity of the function $t \mapsto \sqrt{1-|t|^2}$, yields
    
$$\begin{aligned}
\int_\Omega\bigl[\sqrt{1-|\nabla v|^2}-\sqrt{1-|\nabla u_\rho|^2}\bigr]\,dx
&\leq -\int_\Omega\frac{\nabla u_\rho\cdot\nabla(v-u_\rho)}{\sqrt{1-|\nabla u_\rho|^2}}\,dx\\
&=-\int_\Omega\rho(v-u_\rho)\,dx,
\end{aligned}$$

i.e., that $u_\rho$ is a solution of the variational inequality \eqref{ineq:var1}. Finally, to prove the uniqueness up to constants, we proceed as follows: Rewriting \eqref{ineq:var1}, $u$ solves it if and only if $u$ maximizes over $K$ the functional
$$
J_\rho(w):=\int_\Omega\Big[\sqrt{1-|\nabla w|^2}+\rho\, w\Big]\,dx.
$$
Suppose $u,w\in K$ both maximize $J_\rho$ over $K$, with common maximal value $m$. Since $K$ is convex, the midpoint $z:=\tfrac{u+w}{2}$ belongs to $K$. Since $t\mapsto\sqrt{1-|t|^2}$ is strictly concave on $\{t\in\mathbb{R}^N:|t|\le1\}$, for a.e.\ $x\in\Omega$
$$
\sqrt{1-|\nabla z(x)|^2}\ \ge\ \tfrac12\sqrt{1-|\nabla u(x)|^2}+\tfrac12\sqrt{1-|\nabla w(x)|^2},
$$
with equality at $x$ if and only if $\nabla u(x)=\nabla w(x)$. Integrating over $\Omega$,
\begin{equation}\label{eq:strict-concavity}
\int_\Omega \sqrt{1-|\nabla z|^2}\,dx\ \ge\ \tfrac12\int_\Omega\sqrt{1-|\nabla u|^2}\,dx+\tfrac12\int_\Omega\sqrt{1-|\nabla w|^2}\,dx,
\end{equation}
with equality if and only if $\nabla u=\nabla w$ a.e.\ in $\Omega$. On the other hand, the term $\int_\Omega\rho\,w\,dx$ is linear in $w$, so
$$
\int_\Omega \rho\,z\,dx=\tfrac12\int_\Omega\rho\,u\,dx+\tfrac12\int_\Omega\rho\,w\,dx.
$$
Combining this with \eqref{eq:strict-concavity} yields
$$
J_\rho(z)\ \ge\ \tfrac12 J_\rho(u)+\tfrac12 J_\rho(w)=m,
$$
with equality if and only if $\nabla u=\nabla w$ a.e.\ in $\Omega$. Since $z\in K$ and $m=\max_K J_\rho$, we also have $J_\rho(z)\le m$. Hence $J_\rho(z)=m$, which forces that
$$
\nabla u=\nabla w\quad\text{a.e.\ in }\Omega.
$$
Since $\Omega$ is connected (being convex) and $u,w\in W^{1,\infty}(\Omega)$ satisfy $\nabla(u-w)=0$ a.e., the function $u-w$ is a.e.\ equal to a constant $c_0\in\mathbb{R}$. Thus $u=w+c_0$, proving uniqueness up to an additive constant.
\end{proof}

\subsection{Correspondence of critical points and solutions}
As said before, the idea is to apply the variational framework laid out in the seminal work \cite{BeJeMa14} by Bereanu, Jebelean and Mawhin, applied to \eqref{eq-general}. In that regard, we must first establish the equivalence between critical points in Szulkin's sense of \eqref{functional} and solutions of \eqref{eq-general}, which is our next result:

\begin{theorem}\label{thm:critical-point-equivalence}
Let $u\in K$. Then $u$ is a critical point of $I_\lambda$ (in the sense of Definition \ref{def:crit-pt}) if and only if $u$ is a solution, in the sense of Definition~\ref{def:solution}, of
\begin{equation}\label{eq:PDE-neumann}
-\operatorname{div}\left(\frac{\nabla u}{\sqrt{1-|\nabla u|^2}}\right)=\lambda a(x)g(u)\ \text{ in }\Omega,
\qquad
\frac{\nabla u}{\sqrt{1-|\nabla u|^2}}\cdot \mathbf n=0\ \text{ on }\partial\Omega.
\end{equation}
In particular, any such $u$ automatically satisfies the compatibility condition
\begin{equation}\label{eq:compatibility}
\int_\Omega a(x)g(u)\,dx=0.
\end{equation}
Moreover, $I_\lambda$ is bounded from below and attains its infimum at some $u_0 \in K$, which is a critical point of $I_\lambda$, and hence a solution of \eqref{eq-general}.
\end{theorem}

\begin{proof}
    Let $u \in K$ be a critical point of $I_\lambda$. Then, this implies that $u$ solves
    \begin{equation} \label{ineq:proofteo1}
    \Psi(v) - \Psi(u) - \lambda \int_\Omega a(x)g(u) (v-u) \, dx\geq 0 \quad \text{for all } v\in K. \end{equation}
    Since $u \in K$, one can write
    $$\rho := \lambda a(x)g(u) \in L^\infty(\Omega),$$
    hence, if we prove that $\int_\Omega \rho = 0$, we can conclude this implication by Lemma \ref{lemma:existence} and Lemma \ref{lemma:crit}. To prove it, let $v = u+t \in K$ for $t \in \R$, in \eqref{ineq:proofteo1}. Since $\Psi(v) = \Psi(u)$ (due to $\nabla v = \nabla (u+t) = \nabla u$), one has that
    $-\lambda t \int_\Omega a(x)g(u) \geq 0$,
    for any $t \in \R$. On the one hand, taking $t=1$, this yields
    $$-\lambda \int_\Omega a(x)g(u) \,dx \geq 0,$$
    and, on the other hand, taking $t = -1$, it yields
    $$\lambda \int_\Omega a(x)g(u) \, dx\geq 0,$$
    hence, we must have
    $$\int_\Omega a(x) g(u) \, dx = 0.$$

    Therefore, $u$ itself solves the variational inequality of Lemma \ref{lemma:crit} with this $\rho$, which, combined with Lemma \ref{lemma:existence}, yields that there exists $u_\rho\in W^{2,2}(\Omega)\cap C(\overline\Omega)$ solving $-\mathcal M(u_\rho)=\rho$ with $\frac{\nabla u_\rho}{\sqrt{1-|\nabla u_\rho|^2}} \cdot \mathbf n=0$ on $\partial\Omega$, and by Lemma \ref{lemma:crit}, $u_\rho$ also solves the same variational inequality with this same $\rho$. Since both $u$ and $u_\rho$ solve this inequality with the same $\rho$, by the uniqueness up to a constant in  Lemma \ref{lemma:crit},
    $$u = u_\rho + c_0 \qquad \text{for some } c_0 \in \R.$$
    Then, defining
    $$\mathcal{M}(u) := \operatorname{div}\left( \frac{\nabla u}{\sqrt{1-|\nabla u|^2}} \right),$$ noting that $\mathcal M$ depends only on $\nabla u$, and that $\nabla(u_\rho+c_0)=\nabla u_\rho$, we obtain
    $$-\mathcal M(u) = -\mathcal M(u_\rho+c_0) = -\mathcal M(u_\rho) = \rho = \lambda a(x) g(u).$$
    Likewise, the Neumann condition satisfied by $u_\rho$ transfers to $u$. Hence $u$ itself solves \eqref{eq:PDE-neumann}, allowing us to conclude the first part of the proof.
    \medskip
    Now, let $u\in K$ be a solution of \eqref{eq:PDE-neumann}, that is,
    $$\int_\Omega \frac{\nabla u \cdot \nabla \phi}{\sqrt{1-|\nabla u|^2}}\,dx  = \lambda \int_\Omega a(x) g(u) \phi \, dx\quad \text{for every } \phi \in C^1(\overline{\Omega}).$$
    Taking $\phi = 1$, it yields
    $$\lambda \int_\Omega a(x) g(u) \, dx = 0,$$
    so it satisfies \eqref{eq:compatibility}. Now, to prove that it is a critical point, it follows directly from Lemma \ref{lemma:existence} and Lemma \ref{lemma:crit}, applied to $u$ itself, with $\rho = \lambda a(x) g(u)$.

    \medskip
    Finally, to prove that $I_\lambda$ is bounded from below on $K$,  
    we revisit the estimate obtained in the proof of Lemma \ref{boundness}, specifically \eqref{eq:coerc}, which shows that
$$
I_\lambda(u)\to+\infty
\qquad\text{as }\|u\|_{L^\infty(\Omega)}\to+\infty,
\quad u\in K.
$$
Hence $I_\lambda$ is bounded from below on $K$. Finally, using that $I_\lambda$ satisfies the (PS)-condition (Corollary \ref{cor:PS}) together with the boundedness from below just established, we can use {\cite[Theorem 1.7]{Sz86}}, which yields the existence of $u_0\in K$ such that
$$
I_\lambda(u_0)=\inf_{u\in K} I_\lambda(u),
$$
and $u_0$ is a critical point of $I_\lambda$, finishing the proof.
\end{proof}

\section{Existence of two solutions}\label{sec3}
In this section, we will establish a sufficient condition on the largeness of the parameter $\lambda$ in \eqref{eq-general} to ensure the existence of two distinct non-trivial solutions. Our first result establishes a lower bound for the values $\lambda$ for which problem \eqref{eq-general} admits a non-trivial solution given by global minimization of $I_\lambda$, namely, a ground-state solution.

\begin{theorem} \label{thm:GS}
Let $\lambda \geq0$, $a$ satisfy $(a_*)$ and $(a_+)$, and $g$ satisfy $(g_1)$ and $(g_2)$. Then, there exists $\Lambda>0$ such that, if $\lambda > \Lambda$, then problem \eqref{eq-general} admits a non-trivial solution $u_\lambda^{(l)}$ such that
\begin{equation} \label{eq:inf} I_\lambda(u_\lambda^{(l)}) = \inf_{u \in C(\overline{\Omega})} I_\lambda (u) < 0.\end{equation}
Moreover, if $g$ is odd, $u_\lambda^{(l)}$ can be chosen positive
in $\Omega$; in this case, $-u_\lambda^{(l)}$ is also a solution.

\end{theorem}
\begin{proof}
    By Theorem \ref{thm:critical-point-equivalence}, we know that the infimum \eqref{eq:inf} is achieved at some $u_\lambda^{(l)} \in K$. Therefore, what remains to be proved is that $I_\lambda(u_\lambda^{(l)}) \not= 0.$ 

    \medskip
    To do so, we note that $(a_+)$ implies the existence of some $x_0 \in \{ x \in \Omega \, : \, a(x) > 0 \}$ and a certain $r_0>0$ such that $\overline B_{r_0}(x_0) \subset \{ x \in \Omega \, : \, a(x) > 0 \} \subset \Omega$, and let 
    \begin{equation} \label{def:eta}\eta(x) :=
\begin{cases} \exp\left( \frac{r_0^2}{|x-x_0|^2 - r_{0}^2} \right) &\text{if } x \in B_{r_0}(x_0),\\
0 & \text{if } x \in \Omega \setminus B_{r_0}(x_0).
\end{cases}\end{equation}
With this, we define \begin{equation} \label{def:etabar} \overline \eta (x) := \min\{ r_0, \|\nabla \eta \|^{-1}_{L^\infty(\Omega)}\}\eta(x) ,\end{equation} which, by construction, satisfies $\overline \eta \in K$ and $\operatorname{supp} \overline \eta \subset \{ x \in \Omega \, : \, a(x) > 0 \}$.  Now, writing
$$I_\lambda(\overline{\eta}) = \Psi (\overline{\eta}) - \lambda \int_{\{a > 0 \} } a(x) G(\overline \eta ) \, dx, $$
we note that, choosing
\begin{equation} \label{def:Lambda} \Lambda := \frac{\Psi(\overline \eta)}{\int_{\{a >  0\}} a(x) G(\overline{\eta})\, dx} ,\end{equation}
it is easy to verify that, for every $\lambda > \Lambda $, one has 
$$I_\lambda (\overline \eta) < I_\Lambda (\overline \eta) = 0.$$
From this, the conclusion follows since $I_\lambda $ is bounded from below (cf. Theorem \ref{thm:critical-point-equivalence}) and $I_\lambda(0) = 0$.

Finally, for the last part of the statement, as argued in {\cite[Theorem 2.11]{BoCoZi26}}, it suffices to notice that, if $g$ is odd, then $I_{\lambda}(u)= I_\lambda(|u|) = I_\lambda(-|u|)$ for all $u \in C(\overline{\Omega})$. Hence, replacing the minimizer if necessary, we may assume that
$$
u_\lambda^{(l)}\geq 0 \qquad\text{in }\Omega.
$$
Since $I_\lambda(u_\lambda^{(l)})<0=I_\lambda(0)$, we also have
$u_\lambda^{(l)}\not\equiv0$. By Theorem~\ref{thm:critical-point-equivalence},
$u:=u_\lambda^{(l)}$ is a weak solution of
$$
-\operatorname{div}\left(
\frac{\nabla u}{\sqrt{1-|\nabla u|^2}}
\right)
=
\lambda a(x)g(u)
\qquad\text{in }\Omega.
$$
Now, define $$
h(s):=
\begin{cases}
\dfrac{g(s)}{s}, & s\neq 0,\\ 
g'(0), & s=0.
\end{cases}
$$
Since $g\in C^1(\mathbb R)$ and $g(0)=0$, the function $h$ is
continuous. Hence, setting
$$
c(x):=\lambda a(x)h(u(x)),
$$
we have $c\in L^\infty(\Omega)$ and
$$
\lambda a(x)g(u(x))=c(x)u(x) \quad \text{for a.e. $x\in\Omega$}
$$
Moreover, setting
$$
A(x):=
\frac{1}{\sqrt{1-|\nabla u(x)|^2}}\,I,
$$
where $I$ is the identity matrix of $N \times N$, the equation can be rewritten as
$$
-\operatorname{div}\big(A(x)\nabla u\big)=c(x)u
\qquad\text{in }\Omega.
$$
Then by the regularity result in Lemma \ref{lemma:existence}, there exists $\theta>0$ such
that
$$
\|\nabla u\|_{L^\infty(\Omega)}\leq 1-\theta.
$$
Consequently, $A\in L^\infty(\Omega;\mathbb R^{N\times N})$ is
uniformly elliptic; indeed,
$$
|\xi|^2
\leq A(x)\xi\cdot\xi
\leq
\frac{1}{\sqrt{1-(1-\theta)^2}}\,|\xi|^2
$$
for every $\xi\in\mathbb R^N$ and a.e. $x\in\Omega$. Therefore, the Harnack
inequality for nonnegative solutions (cf. {\cite[Theorem 8.20]{GT01}}) implies that either
$u\equiv0$ in $\Omega$ or $u>0$ in $\Omega$. Since $u\not\equiv0$, the former alternative is impossible, and hence
$$
u_\lambda^{(l)}>0\qquad\text{in }\Omega,
$$
finishing the proof.
\end{proof}

Having produced one non-trivial solution $u_\lambda^{(l)}$ as a global minimizer with negative energy, we now look for a \emph{second}, distinct, non-trivial solution with positive energy, which we obtain via a non-smooth mountain pass argument. 

\begin{theorem} \label{thm:MP}
    Let $\lambda >\Lambda$, where $\Lambda$ is as in Theorem \ref{thm:GS}, $a$ satisfy $(a_*)$ and $(a_+)$, and $g$ satisfy $(g_1)$, $(g_2)$ and $(g_3)$. Then, problem \eqref{eq-general} has a non-trivial solution $u_\lambda^{(s)}$ such that
    $$I_\lambda(u_\lambda^{(s)}) > 0.$$
    Moreover, if $g$ is odd, $u_\lambda^{(s)}$ can be chosen positive in $\Omega$.
\end{theorem}
\begin{proof}
To begin the proof, note that $I_\lambda(0) = 0$ and, as established in Corollary \ref{cor:PS}, $I_\lambda$ satisfies the (PS)-condition. Now, let us consider $\Lambda$ as defined in \eqref{def:Lambda}, $\lambda > \Lambda$ and $u_\lambda^{(l)}$ be the solution given by Theorem \ref{thm:GS}. The idea of the proof is to apply Szulkin's non-smooth Mountain Pass theorem (cf. \cite[Theorem 3.2]{Sz86}). I.e., it suffices to show that there exists $\alpha >0$ and $\rho \in (0, \|u_\lambda^{(l)}\|_{L^{\infty}(\Omega) }) $ such that,
    \begin{equation}\label{eq:MP} I_\lambda(u) \geq \alpha , \quad \text{for all } u \in K \text{ s.t. } \|u\|_{L^{\infty}(\Omega)} = \rho.\end{equation}
    To do so, let us consider for $u \in K$, 
    $$u = \bar u + v, \quad \text{where } \bar u := \frac{1}{|\Omega|}\int_\Omega u \, dx.$$
    From it, it follows that $v = u - \bar u \in K$ satisfy \begin{equation} \label{eq:v} \int_\Omega v \, dx= 0,\end{equation} and also
    $$|\nabla u| = |\nabla v| \quad \text{a.e. in } \Omega.$$
    Hence, $\Psi(u) = \Psi(v)$. Additionally, using the elementary inequality
    $$1-\sqrt{1-s^2} \geq \frac{s^2}2 \quad \text{for all } |s| \leq 1,$$
   and that \eqref{eq:v} allows us to use the Poincaré inequality on $v$, one has that 
    \begin{equation}\label{eq:Psiest} \Psi(u) \geq C_p \int_\Omega v^2 \, dx,\end{equation}
    for some $C_p>0$. Using this, we deduce
    \begin{equation} \label{eq:estI1}\begin{aligned} I_\lambda(u) &= \Psi(u) - \lambda \int_\Omega a(x) G(u) \, dx \\
    &\geq C_p \|v\|_{L^2(\Omega)}^2 - \lambda G(\bar u) \int_{\Omega} a(x) \, dx - \lambda \int_{\Omega} a(x) \big[G(\bar u + v) - G(\bar u ) \big] \, dx.\end{aligned}\end{equation}
    Therefore, what we want is to estimate from below \eqref{eq:estI1} in terms of $v$, that is, the zero-average oscillatory part of $u$. To do so, let
    \begin{equation} \label{eq:omegarho}
        \omega(\rho):=
        \sup_{|s|\leq\rho}|g'(s)|,
    \end{equation}
    and note that, by $(g_1)$ and $(g_3)$, one has that,
    \begin{equation}\label{eq:omega}
        \omega(\rho)\rightarrow0
        \qquad\text{as }\rho\to0.
    \end{equation}
    Now, assume that $\|u\|_{L^\infty(\Omega)}\leq\rho$ for some $\rho>0$. Then
    \begin{equation} \label{eq:estrho1}
        |\bar u|\leq\rho,
        \qquad
        |\bar u+tv|\leq\rho
        \quad\text{for every }t\in[0,1].
    \end{equation}
    To estimate from below \eqref{eq:estI1} in terms of $v$, since we already know that $-\lambda G(\bar u ) \int_\Omega a(x) \,dx \geq 0$ (by $(a_\ast)$ and $(g_1)$), we ought to estimate the difference
    $$G(\bar u +v) - G(\bar u),$$
    which we do by means of Taylor's formula applied to $G\in C^2(\R)$ at $\bar u$:
    \begin{equation} \label{eq:Taylor1}
        G(\bar u+v)
        =
        G(\bar u)+g(\bar u)v+R(\bar u, v),
    \end{equation}
    where $R$ is the integral rest, that is,
    $$
        R(\bar u, v)
        =
        v^2\int_0^1(1-t)
        g'(\bar u+tv)\,dt.
    $$
    Note that, by \eqref{eq:estrho1} and \eqref{eq:omegarho}, is easy to see that
    $$
        |R|
        \leq
        \frac{\omega(\rho)}{2}v^2.
    $$
    Substituting this into \eqref{eq:Taylor1}, we deduce
    \begin{equation}\label{eq:Taylor}
        \left|
        G(\bar u+v)-G(\bar u)-g(\bar u)v
        \right|
        \leq
        \frac{\omega(\rho)}{2}v^2.
    \end{equation}
    and, therefore, 
    $$
    \begin{aligned}
        \left|
        \int_\Omega
        a(x)\big[G(\bar u+v)-G(\bar u)\big]\,dx
        \right|
        \leq
        |g(\bar u)|
        \left|\int_\Omega a(x)v\,dx\right|
        +
        \frac{\|a\|_\infty}{2}
        \omega(\rho)\|v\|_{L^2(\Omega)}^2.
    \end{aligned}
    $$
    Since $\int_\Omega v\,dx=0$, we have
    $$
        \int_\Omega a(x)v\,dx
        =
        \int_\Omega (a(x)-\bar a)v\,dx, \quad \text{where } \bar a := \frac{1}{|\Omega|}\int_\Omega a(x) \, dx, 
    $$
    and hence, by the Cauchy--Schwarz inequality,
    $$
        \left|\int_\Omega a(x)v\,dx\right|
        \leq
        \|a-\bar a\|_{L^2(\Omega)}
        \|v\|_{L^2(\Omega)}.
    $$
    Thus, 
    \begin{equation}\label{eq:error}
        \left|
        \int_\Omega
        a(x)\big[G(\bar u+v)-G(\bar u)\big]\,dx
        \right|
        \leq |g(\bar u)|\,
        \|a-\bar a\|_{L^2 (\Omega)} \|v\|_{L^2(\Omega)}
        +
        \frac{\|a\|_{L^\infty (\Omega) }}{2}\omega(\rho)\|v\|_{L^2(\Omega)}^2.
    \end{equation}
    Combining \eqref{eq:estI1} and \eqref{eq:error}, we obtain
    \begin{equation}\label{eq:Ilambda-est}
    \begin{aligned}
        I_\lambda(u)
        \geq{}
        \left(C_p-\frac{\|a\|_{L^\infty (\Omega) }}{2}\lambda\omega(\rho)\right)
        \|v\|_{L^2(\Omega)}^2
        +\lambda |A|G(\bar u)
        -\|a-\bar a \|_{L^2 (\Omega)}\lambda |g(\bar u)|
        \|v\|_{L^2(\Omega)},
    \end{aligned}
    \end{equation}
    where $A := \int_\Omega a(x) dx$. By \eqref{eq:omega}, we can choose $\rho_1>0$ sufficiently small
    such that
    $$
        \lambda\frac{\|a\|_{L^\infty (\Omega) }}{2}\omega(\rho)\leq\frac{C_p}{2}
        \qquad\text{for }0<\rho\leq\rho_1.
    $$
    Hence,
    \begin{equation}\label{eq:Ilambda-est2}
    I_\lambda(u)
    \geq
    \frac {C_p} 2\|v\|_{L^2(\Omega)}^2
    +\lambda |A|G(\bar u)
    -\|a-\bar a \|_{L^2(\Omega)}\lambda |g(\bar u)|
    \|v\|_{L^2(\Omega)}.
    \end{equation}
    Applying Young's inequality
    $$cd \leq \varepsilon c^2 + \frac{d^2}{4\varepsilon} \quad \text{ for all } \varepsilon>0$$
    on the third term, with $c = \|v\|_{L^2(\Omega)}$, $d = \|a-\bar a \|_{L^2(\Omega)}\lambda |g(\bar u)|$ and $\varepsilon = \frac{C_p}{4}$, we get
    $$
        \|a-\bar a\|_{L^2(\Omega) }\lambda |g(\bar u)|
        \|v\|_{L^2(\Omega)}
        \leq
        \frac{C_p} 4\|v\|_{L^2(\Omega)}^2
        + \frac{1}{C_p} \|a-\bar a\|^2_{L^2(\Omega) }\lambda^2 g(\bar u)^2.
    $$
 Therefore,
    \begin{equation}\label{eq:Ilambda-est3}
        I_\lambda(u)
        \geq
        \frac {C_p}4\|v\|_{L^2(\Omega)}^2
        +\lambda |A|G(\bar u)
        -C_a\lambda^2g(\bar u)^2,
    \end{equation}
    with $C_a := \frac{1}{C_p} \|a-\bar a\|^2_{L^2(\Omega) }.$
    By $(g_3)$, for $\lambda>0$ fixed, we can choose $\rho_2>0$ sufficiently small such that
    $$
        C_a\lambda^2g(s)^2
        \leq
        \frac{\lambda |A|}{2}G(s)
        \qquad\text{for } |s|\leq\rho_2.
    $$
    Thus, taking
    $$
        0<\rho<\min\{\rho_1,\rho_2, \|u_\lambda^{(l)}\|_{L^{\infty}(\Omega)}\},
    $$
    and recalling that $|\bar u|\leq\|u\|_\infty=\rho$, we deduce from
    \eqref{eq:Ilambda-est3} that
    \begin{equation}\label{eq:Ilambda-est4}
        I_\lambda(u)
        \geq
        \frac {C_p}4\|v\|_{L^2(\Omega)}^2
        +\frac{\lambda |A|}{2}G(\bar u).
    \end{equation}
    It remains to obtain a uniform positive lower bound. We distinguish
    two cases: if $|\bar u|\geq\rho/2$, then, since $G(s)>0$ for every $s\neq0$,
    $$
        \gamma_\rho:=
        \min_{\rho/2\leq |s|\leq\rho}G(s)= \min \{G(\rho/2), G(-\rho/2) \}>0.
    $$
    Thus, by \eqref{eq:Ilambda-est4},
    \begin{equation}\label{eq:case1}
        I_\lambda(u)
        \geq
        \frac{\lambda |A|}{2}\gamma_\rho>0.
    \end{equation}
    On the other hand, suppose that $|\bar u|<\rho/2$. Since
    $$
        \|u\|_{L^\infty(\Omega)}=\rho,
    $$
    there exists $x_0\in\overline{\Omega}$ such that
    $$
        |u(x_0)|=\rho.
    $$
    Consequently,
    $$
        |v(x_0)|
        =
        |u(x_0)-\bar u|
        \geq
        |u(x_0)|-|\bar u|
        >
        \frac{\rho}{2}.
    $$
    Since $\|\nabla v\|_\infty\leq1$, it follows that
    $$
        |v(x)|\geq\frac{\rho}{4}
    $$
    for every $x\in B_{\rho/4}(x_0)\cap\Omega$. Therefore, for $\rho>0$
    sufficiently small, there exists $c_\Omega>0$ such that
    $$
        \|v\|_{L^2(\Omega)}^2
        \geq
        c_\Omega\rho^{N+2}.
    $$
    Hence, by \eqref{eq:Ilambda-est4},
    \begin{equation}\label{eq:case2}
        I_\lambda(u)
        \geq
        \frac{C_p \ c_\Omega}{4}\rho^{N+2}>0.
    \end{equation}
    Combining \eqref{eq:case1} and \eqref{eq:case2}, we conclude that
    $$
        I_\lambda(u)\geq\alpha>0
    $$
    for every $u\in K$ satisfying
    $\|u\|_{L^\infty(\Omega)}=\rho$, where
    $$
        \alpha:=
        \min\left\{
        \frac{\lambda |A|}{2}\gamma_\rho,
        \frac{C_p \ c_\Omega}{4}\rho^{N+2}
        \right\}.
    $$
    This proves \eqref{eq:MP}, and therefore, the existence of a critical point $u_\lambda^{(s)}$ characterized by its energy level
    $$c_\lambda = \inf_{f\in \Gamma} \,\sup_{t \in [0,1]} I_\lambda(f(t)),$$
    where $\Gamma := \{ f \in C\big ([0,1], C(\overline{\Omega}) \big ) \, : \, f(0) =0, \, f(1) = u_\lambda^{(l)} \}.$
    Moreover, using the Proposition \ref{prop:abstract-positivity} with $p(u) = |u|$, one has that, if $g$ is odd (and, consequently, $u_\lambda^{(l)}$ is positive), there exist a certain $u_\lambda^{(s)} \in K^+ := \{ u \in K \, : \, u \geq 0 \quad \text{in } \Omega \} = p(K)$ such that 
    $$c_\lambda = I_\lambda(u_\lambda^{(s)}).$$ Then, reasoning as at the end of the proof of Theorem \ref{thm:GS}, one can conclude that $u_\lambda^{(s)} > 0$ in $\Omega$, finishing the proof.
    
\end{proof}

Combining the two theorems above -- the first providing a negative-energy global minimizer, the second a positive-energy mountain pass critical point -- we obtain the following multiplicity result, which is the main outcome of this section.

\begin{corollary}
    Let $a$ satisfy $(a_*)$ and $(a_+)$, and $g$ satisfy $(g_1)$, $(g_2)$ and $(g_3)$. Then, there exists a $\Lambda >0$ such that, if $\lambda > \Lambda $, problem \eqref{eq-general} admits two non-trivial solutions $u_{\lambda}^{(s)}$ and $u_\lambda^{(l)}$ satisfying,
    $$I_\lambda(u_{\lambda}^{(l)}) < 0 < I_{\lambda}(u_\lambda^{(s)}).$$
    Moreover, if $g$ is odd, then both solutions can be chosen positive.
\end{corollary}

\section{Asymptotic behaviour as $\lambda \to +\infty$} \label{sec4}

Throughout this section, we assume that $(a_*)$, $(a_+)$, $(g_1)$, $(g_2)$, $(g_3)$ hold. The compactness result and the analysis of the ground-state's limiting profile below hold under the general assumptions $(g_1)$-$(g_2)$-$(g_3)$. However, the asymptotic results corresponding to the family of mountain-pass solutions $u_\lambda^{(s)} $ will require us to restrict our attention to the model case
$$g(u)=|u|^{p-2}u,\qquad p>2.$$
Whenever this specific choice of $g$ is required, it will be stated explicitly.

Our first result consists of proving that all solutions lie within a compact subset of $K$, independent of $\lambda$, and characterized by the compatibility condition \eqref{eq:compatibility}. Namely, 
\begin{equation} \label{eq:Ka} K_a:= \{ u \in K \, : \, \int_\Omega a (x)g(u) \, dx = 0 \}.\end{equation}
This result is of particular importance to our asymptotic analysis, since the compactness of $K_a$ allows us to extract uniformly convergent subsequences of solutions $u_\lambda$ of \eqref{eq-general}, as $\lambda \to + \infty$, which is a property that generally lacks on the Neumann boundary condition case, unlike in the Dirichlet case (cf. \cite{BeJeMa14, BoCoZi26}). 

\begin{lemma} \label{lemma:compactness}
    Consider that $(a_*)$, $(a_+)$, $(g_1)$, $(g_2)$ and $(g_3)$ hold. Then, $K_a$ is a compact subset of $C(\overline{\Omega})$.
\end{lemma}

\begin{proof}
    To begin with the proof, one can easily verify that $K_a$ is a closed subset of $C(\overline{\Omega})$: $K$ is closed in $C(\overline{\Omega})$, the map $F(u) := \int_\Omega a(x) g(u) \, dx$ is continuous on $C(\overline{\Omega})$, and $K_a = K \cap F^{-1}(\{0\})$ is the intersection of two closed sets.

    \medskip

    To prove the compactness, note that, for every $u \in K_a$, one can decompose it as
    $$u = \bar u + v,$$
    where $\bar u:= \frac{1}{|\Omega|} \int_\Omega u \, dx$ as usual, and $v\in K$ satisfy $\int_\Omega v \, dx = 0$. Let $\{u_n\}_{n \in \N} \subset K_a$, and consider the same decomposition, that is, for every $n \in \N$, consider \begin{equation} \label{eq:un} u_n = \bar u_n + v_n. \end{equation} First, we claim that $\{v_n\}_{n \in \N} \subset K$ is uniformly bounded.

\medskip

    To prove it, it suffices to see that, for every function $v \in K$ satisfying 
    $$\int_\Omega v \, dx = 0,$$
    due to the connectedness of $\Omega$, one has the existence of a certain $x_0 \in \Omega$ such that $v(x_0) = 0$. From this, the $1$-Lipschitz property of $v$ yields, 
    $$|v(x)| = |v(x) - v(x_0)| \leq |x-x_0| \leq D,$$
    for every $x \in \Omega$, where $D := \operatorname{diam}(\Omega)$. In other words, $\|v\|_{L^\infty(\Omega)} \leq D.$ In particular, this gives us a uniform bound for the sequence $\{v_n\}_{n \in \N}$. Now, since $\{v_n\}_{n \in \N} \subset K$, and is uniformly bounded, a simple application of the Ascoli-Arzelà theorem gives us (up to a subsequence) that
    $$v_n \to v_\infty \in K \quad \text{uniformly as } n \to + \infty, $$
    and, by the continuity of the integral, we have that $\int_\Omega v_\infty \, dx = 0$. 

    Now, in view of \eqref{eq:un}, one would have the compactness of $K_a$ if we prove that $\{\bar u_n\}_{n \in \N} \subset \R$ is bounded. Reasoning by contradiction, assume that $\bar u_n\to+\infty$
(the case $\bar u_n\to-\infty$ is analogous). Since
$$
u_n=\bar u_n+v_n,
\qquad
\|v_n\|_{L^\infty(\Omega)}\le D,
$$
we have
$$
\frac{u_n(x)}{\bar u_n}
=
1+\frac{v_n(x)}{\bar u_n}
\to1
$$
uniformly in $x\in\Omega$. Therefore, by $(g_2)$,
$$
g(u_n(x)) = c\,u_n(x)^{p-1}(1+o(1)) = c\,\bar u_n^{p-1}(1+o(1))
$$
uniformly in $x\in\Omega$. Consequently,
$$
0 = \int_\Omega a(x)g(u_n(x))\,dx = c\,\bar u_n^{p-1} \left( \int_\Omega a(x)\,dx+o(1) \right),
$$
which is impossible for $n$ sufficiently large, since $c>0$ and $\int_\Omega a(x)\,dx<0$. Therefore, $\{ \bar u _n\}_{n \in \N}$ must be bounded, and consequently, up to a subsequence, $\bar u _n \to \bar u_\infty \in \R$. The latter implies that
    $$u_n = \bar u_n + v_n \to u_\infty = \bar u_\infty + v_\infty \quad \text{as } n \to +\infty,$$
    up to a subsequence, and by $K_a$ being closed, $u_\infty \in K_a$. Since $\{u_n\}_{n \in \N}$ was arbitrary, the proof is completed.

\end{proof}
Now, with this compactness result, we can proceed with the asymptotic analysis. Our first result can be seen as a direct consequence of {\cite[Proposition 2.17]{BoCoZi26}}, thanks to the compactness of $K_a$, and we include the proof for the sake of completeness.

\begin{proposition} \label{prop:asGS}
    Assume that $(a_*)$, $(a_+)$, $(g_1)$, $(g_2)$ and $(g_3)$ hold. Let $M >0$ be such that, for every $\lambda \geq M$, $u_\lambda$ is a solution of \eqref{eq-general} satisfying
    \begin{equation} \label{eq:infchar}I_\lambda(u_\lambda) = \inf_{u \in C(\overline{\Omega})} I_\lambda(u). \end{equation}
    Then, up to a subsequence,
    $$u_\lambda \to u_\infty \quad \text{uniformly in $\overline \Omega $ as } \lambda \to + \infty,$$
    where $u_\infty \in K_a$ solves the following maximum problem

    \begin{equation} \label{eq:varmax}
    \max_{u \in K} \int_\Omega a(x) G(u) \, dx,\end{equation}
    and, $u_\infty \not \equiv 0$. In particular, the above conclusion holds for any family of ground-state solutions $\{u_\lambda^{(l)} \}_{\lambda \geq M}$ as in Theorem \ref{thm:GS}.
\end{proposition}

\begin{proof}
    By the compactness of $K_a$ (cf. Lemma \ref{lemma:compactness}) and the fact that, by $u_\lambda$ being a solution of \eqref{eq-general} one has that $u_\lambda \in K_a$ for every $\lambda > M$, we know that there exists a function $u_\infty \in K_a$ such that, up to a subsequence, $u_\lambda \to u_\infty$ in $C(\overline{\Omega})$ as $\lambda \to +\infty$. Reasoning by contradiction, assume that $u_\infty$ does not solve \eqref{eq:varmax}, and let $u^*\in K$ be a competitor such that
    $$\int_\Omega a(x) G(u^*) \, dx > \int_\Omega a(x) G(u_\infty) \, dx.$$
    Now, using that $u_\lambda$ is characterized by \eqref{eq:infchar}, we obtain
    $$0 \leq I_\lambda(u^*) - I_\lambda(u_\lambda)  = \Psi (u^*) - \Psi (u_\lambda) - \lambda \left( \int_\Omega a(x) \big[ G(u^*) - G(u_\lambda) \big ] \, dx \right).$$
Since $0\leq\Psi\leq|\Omega|$ on $K$, the first two terms are uniformly bounded in $\lambda$. Moreover, by $(g_1)$ and how we defined $u^*$, we have that
    $$\int_\Omega a(x) \big [ G(u^*) - G(u_\lambda) \big ] \, dx \to \int_\Omega a(x) \big [ G(u^*) - G(u_\infty) \big ] \, dx > 0 \quad \text{ as } \lambda \to +\infty.$$
    Combining this, we have
    $$0 \leq I_\lambda (u^*) - I_\lambda(u_\lambda) = O(1) -  \lambda \int_\Omega a(x) \big [ G(u^*) - G(u_\lambda) \big ] \, dx \to -\infty  \quad \text{as } \lambda \to +\infty,$$
    which is a contradiction. Therefore, $u_\infty$ must solve \eqref{eq:varmax}. Now, to see that $u_\infty \not \equiv 0$, it suffices to see that, for any $\mu > \Lambda$ where $\Lambda$ is as defined in Theorem \ref{thm:GS}, $$\int_\Omega a(x)G(u_\mu^{(l)}) \, dx >0,$$ 
    where $u_\mu^{(l)} \in K_a$ is the non-trivial ground-state solution given by said theorem, that is,
    $$\inf_{u \in C(\overline \Omega)}I_\mu(u) = I_\mu(u_\mu^{(l)} ).$$ Hence the conclusion follows.
\end{proof}

\begin{remark}
We stress that $(g_2)$ can be weakened in Lemma~\ref{boundness},
Theorem~\ref{thm:critical-point-equivalence}, Lemma~\ref{lemma:compactness}, and Proposition~\ref{prop:asGS}.
More precisely, it is enough to assume, in place of $(g_2)$, that
\begin{enumerate}
    \item[(i)]
    $\displaystyle G(s)\to+\infty$ as $|s|\to+\infty$;

    \item[(ii)]
     $\displaystyle \lim_{\substack{\omega\to1\\ |s|\to+\infty}}
\frac{g(\omega s)}{g(s)}
=1.$
    \item[(iii)]
    $\displaystyle
    \frac{|g(s)|}{G(s)}\to0
    $ as $|s|\to+\infty$.
\end{enumerate}
Indeed, these are the only asymptotic properties of $(g_2)$
used in the relevant proofs. In the coercivity argument in Lemma \ref{boundness}, if
$M_n\to+\infty$, condition (ii) yield, writing $\omega_n = \frac{M_n}{M_n-D}$,
$$
g(M_n)= g(\omega_n (M_n-D))  =g(M_n-D)(1+o(1)),
$$
while conditions (ii) and (iii) combined give
$$
G(M_n)-G(M_n-D)
\le
D\sup_{t\in[M_n-D,M_n]}g(t)
=o\bigl(G(M_n-D)\bigr).
$$
Together with (i) and $A^- > A^+$, this implies
$I_\lambda(u_n)\to+\infty$. The case $M_n\to-\infty$ is
analogous. Likewise, in the proof of Lemma~\ref{lemma:compactness}, suppose that
$\bar u_n\to+\infty$ and write
$$
u_n=\bar u_n+v_n,
\qquad
\|v_n\|_{L^\infty(\Omega)}\le D.
$$
For every $x\in\Omega$, set
$$
\omega_n(x):=1+\frac{v_n(x)}{\bar u_n},
$$
so that
$$
u_n(x)=\omega_n(x)\bar u_n.
$$
Since
$$
\sup_{x\in\Omega}|\omega_n(x)-1|
\le \frac{D}{\bar u_n}\to0,
$$
condition (ii) yields
$$
\frac{g(u_n(x))}{g(\bar u_n)}
=
\frac{g(\omega_n(x)\bar u_n)}{g(\bar u_n)}
\to1 \qquad \text{uniformly in $x\in\Omega$.}
$$
Hence
$$
g(u_n(x))=g(\bar u_n)(1+o(1))
$$
uniformly in $x\in\Omega$, and consequently
$$
\int_\Omega a(x)g(u_n(x))\,dx
=
g(\bar u_n)
\left(
\int_\Omega a(x)\,dx+o(1)
\right).
$$
Since $\bar u_n>0$ for $n$ sufficiently large, $(g_1)$ gives
$g(\bar u_n)>0$, while
$$
\int_\Omega a(x)\,dx+o(1)<0.
$$
Thus the right-hand side is strictly negative for $n$ sufficiently large,
contradicting the compatibility condition
$$
\int_\Omega a(x)g(u_n(x))\,dx=0.
$$
The case $\bar u_n\to-\infty$ is analogous. Finally, notice that $(g_2)$ implies (i), (ii) and (iii), but the converse does not hold. For instance, $g(s)=|s|^{p-2}s\log(1+|s|)$ satisfies (i), (ii) and (iii), but not $(g_2)$. We however retain the simpler asymptotic assumption $(g_2)$ throughout
the paper for readability and consistency with the model case $g(u)= |u|^{p-2} u$ with $p>2$.
\end{remark}
In what follows, we examine the asymptotic behaviour of the min-max solutions obtained via the mountain pass theorem (see Theorem \ref{thm:MP}), namely $u_\lambda^{(s)}$. To this end, we reduced the problem to the model case $g(u) = |u|^{p-2} u$ with $p>2$, which enabled us to determine the asymptotic limit of $u_\lambda^{(s)}$ as $\lambda \to +\infty$, and also to establish a decay rate for its corresponding energy level $c_\lambda$. This is summarized in the following result.

\begin{proposition}
\label{prop:MP-asymptotics}
Assume $(a_\ast)$ and $(a_+)$, and consider the model
nonlinearity
$$
g(u)=|u|^{p-2}u,\qquad p>2.
$$
Let $\lambda > \Lambda$ (as in Theorem \ref{thm:GS}),  $u_\lambda^{(l)}$ be the ground-state solution and $u_\lambda^{(s)}$ be the mountain-pass solution characterized by its energy level,
$$c_\lambda:= \inf_{f\in \Gamma} \sup_{t \in [0,1]} I_\lambda(f(t)), \quad \Gamma := \{ f \in C\big ([0,1], C(\overline{\Omega}) \big ) \, : \, f(0) =0, \, f(1) = u_\lambda^{(l)} \}.$$
Then
\begin{equation} \label{eq:clambdau0}c_\lambda=O\left(\lambda^{-\frac{2}{p-2}}\right)
\quad\text{and}\quad
u_\lambda^{(s)}\rightarrow0
\quad\text{in }C(\overline{\Omega})
\end{equation}
as $\lambda\to+\infty$.
\end{proposition}
\begin{proof}
    Let us consider $f(t) = t u_\lambda^{(l)}$. Clearly, $f \in \Gamma $. Therefore, by the characterization of $c_\lambda$, one has that
    $$c_\lambda \leq \max_{t \in [0,1]} I_\lambda(t u_\lambda^{(l)}).$$
    Now, we turn ourselves to study
    $$t\mapsto I_\lambda(tu_\lambda^{(l)}) = \int_\Omega 1 - \sqrt{1-t^2 |\nabla u_\lambda^{(l)}|^2} \, dx - t^p  \frac{\lambda}{p} \int_\Omega a(x) |u_\lambda^{(l)}|^p \, dx.$$
    To do so, we first use the following elementary inequality involving the map $s\mapsto 1-\sqrt{1-s^2} $,
    $$1-\sqrt{1-(ts)^2} \leq t^2 (1 -\sqrt{1-s^2}) \quad \text{for every } t\in [0,1],$$
    which implies, 
    $$\Psi(t u_\lambda^{(l)}) \leq t^2 \Psi(u_\lambda^{(l)}) \quad \text{for every } t\in [0,1].$$
    Thus, using that $\Psi(u) \leq |\Omega|$ for every $u \in K$, one has that
    $$\max_{t \in [0,1]} I_\lambda(t u_\lambda^{(l)}) \leq \max_{t \in [0,1]} \left( t^2 |\Omega | - t^p \frac{\lambda}{p}  \, \int_\Omega a(x) |u_\lambda^{(l)}|^p \, dx \right).$$
    Now, we obtain a uniform positive lower bound for $\int_\Omega a(x)|u_\lambda^{(l)}|^p\,dx$. Recall from the proof of Theorem~\ref{thm:GS} the fixed test function $\bar \eta \in K$ (cf. \eqref{def:eta} and \eqref{def:etabar}), independent of $\lambda$, with $\operatorname{supp} \bar \eta \subset \{a>0\}$, and set
$$\kappa := \int_{\{a>0\}} a(x)\, G(\bar \eta) \, dx > 0, \qquad \Lambda = \frac{\Psi(\bar \eta)}{\kappa}.$$
For every $\lambda > \Lambda$, since $u_\lambda^{(l)}$ minimizes $I_\lambda$ over $K$,
$$I_\lambda(u_\lambda^{(l)}) \leq I_\lambda(\bar\eta) = \Psi(\bar\eta) - \lambda \kappa.$$
On the other hand, since $\Psi \geq 0$ on $K$ and $G(s) = |s|^p/p$,
$$I_\lambda(u_\lambda^{(l)}) = \Psi(u_\lambda^{(l)}) - \frac{\lambda}{p} \int_\Omega a(x) |u_\lambda^{(l)}|^p \, dx \geq -\frac{\lambda}{p} \int_\Omega a(x) |u_\lambda^{(l)}|^p \, dx.$$
Combining these two estimates,
$$\int_\Omega a(x) |u_\lambda^{(l)}|^p \, dx \geq p\kappa - \frac{p\,\Psi(\bar \eta)}{\lambda}.$$
Hence, for every $\lambda \geq  \max\{\Lambda,\ 2\Psi(\bar\eta)/\kappa\}$,
$$\int_\Omega a(x) |u_\lambda^{(l)}|^p \, dx \geq \frac{p\kappa}{2} =: A_0 > 0.$$Hence, taking $\lambda$ large enough, one has 
    \begin{equation} \label{eq:clambda}
    c_\lambda \leq \max_{t \in [0,1]} I_\lambda(t u_\lambda^{(l)}) \leq \max_{t \in [0,1]} \left( t^2 |\Omega | - t^p\frac{\lambda}{p}  A_0 \right),\end{equation}
    which, by differentiating the right side, yields that the maximum is reached in
    $$t_\lambda = \left(\frac{2 |\Omega|}{A_0 \lambda}\right)^{\frac{1}{p-2}}. $$
    Putting this in \eqref{eq:clambda} yields
    $$c_\lambda \leq C \lambda^{\frac{-2}{p-2}},$$
    hence proving the first part of \eqref{eq:clambdau0}.

    \medskip
    Now focus on proving the last part of \eqref{eq:clambdau0}. See that, by $u_\lambda^{(s)}$ being a solution of \eqref{eq-model}, 
    \begin{equation} \label{eq:weaksol}\int_\Omega \frac{ \nabla u_\lambda^{(s)} \cdot \nabla \phi }{\sqrt{1 - |\nabla u_\lambda^{(s)}|^2}} \,dx = \lambda \int_{\Omega} a(x) |u_{\lambda}^{(s)}|^{p-2} u_\lambda^{(s)} \phi\, dx \quad \text{for every } \phi \in C^1(\overline\Omega).\end{equation}
    Hence, testing with $\phi = u_{\lambda}^{(s)}$, it yields
    $$\int_\Omega \frac{ |\nabla u_\lambda^{(s)}|^2  }{\sqrt{1 - |\nabla u_\lambda^{(s)}|^2}} \,dx = \lambda \int_{\Omega} a(x) |u_{\lambda}^{(s)}|^{p} \, dx, $$
    which, combined with the expression of $c_\lambda$, yields
    $$0<c_\lambda = \Psi(u_\lambda^{(s)}) - \frac{1}{p} \int_\Omega \frac{ |\nabla u_\lambda^{(s)}|^2  }{\sqrt{1 - |\nabla u_\lambda^{(s)}|^2}} \,dx, $$
    from which, using that $\Psi(u_\lambda^{(s)}) \leq| \Omega|$, we deduce that
    $$ \left \| \frac{ |\nabla u_\lambda^{(s)}|^2  }{\sqrt{1 - |\nabla u_\lambda^{(s)}|^2}} \right \|_{L^1(\Omega)} \leq C,$$
    for some $C>0$ uniform on $\lambda$. Now, coming back to \eqref{eq:weaksol}, from the LHS one has that
    $$\left | \int_\Omega \frac{\nabla u_\lambda^{(s)} \cdot \nabla \phi}{\sqrt{1-|\nabla u_\lambda^{(s)}|^2}} \, dx  \right | \leq \| \nabla \phi\|_{L^{\infty}(\Omega)}  \int_\Omega \frac{|\nabla u_\lambda^{(s)}|}{\sqrt{1-|\nabla u_\lambda^{(s)}|^2}} \, dx,  $$
    from which, using the elementary inequality
    $$\frac{s}{\sqrt{1-s^2}} \leq 2 \left( \frac{s^2}{\sqrt{1-s^2}} + 1 \right) \quad \text{for every } s\in[0,1),$$
    we deduce
    $$\left | \int_\Omega \frac{\nabla u_\lambda^{(s)} \cdot \nabla \phi}{\sqrt{1-|\nabla u_\lambda^{(s)}|^2}} \, dx  \right |  \leq C \|\nabla \phi\|_{L^{\infty}(\Omega)} \quad \text{for every } \phi \in C^1(\overline \Omega). $$
    Using this in \eqref{eq:weaksol} and dividing by $\lambda$ yields that
    $$\left | \int_\Omega a(x) |u_\lambda^{(s)}|^{p-2} u_\lambda^{(s)} \phi  \, dx\right | \leq \frac{C}{\lambda} \|\nabla \phi\|_{L^\infty(\Omega)} \quad \text{for every } \phi \in C^1(\overline \Omega).$$
    Then, given the compactness of $K_a$, and that $u_\lambda^{(s)} \in K_a$, we have that, up to a subsequence, $$u_\lambda^{(s)} \to u_\infty^{(s)} \quad \text{as } \lambda \to +\infty.$$
    Hence,
    $$\left | \int_\Omega a(x) |u_\infty^{(s)}|^{p-2} u_\infty^{(s)} \phi \, dx \right | = 0 \quad \text{for every } \phi \in C^1(\overline{\Omega}),$$
    which implies,
    $$ a(x) |u_\infty^{(s)}|^{p-2} u_\infty^{(s)} = 0 \quad \text{a.e. in } \Omega,$$

Hence $u_\infty^{(s)}=0$ a.e. on $E:=\{x\in\Omega:a(x)\neq0\}$. We remark that, due to $(a_*)$, $|E| > 0$.

To conclude, fix $\varepsilon>0$. By uniform convergence, for $\lambda$ sufficiently large we have, along the subsequence under consideration,
$$|u_\lambda^{(s)}|\leq\varepsilon\qquad\text{a.e. on }E.$$
Taking $\phi=(u_\lambda^{(s)}-\varepsilon)^+ \in W^{1,2}(\Omega)$ in the weak formulation gives,
$$\int_{\{u_\lambda^{(s)}>\varepsilon\}}
\frac{|\nabla u_\lambda^{(s)}|^2}{\sqrt{1-|\nabla u_\lambda^{(s)}|^2}}\,dx
=\lambda\int_\Omega a(x)|u_\lambda^{(s)}|^{p-2}u_\lambda^{(s)}(u_\lambda^{(s)}-\varepsilon)^+\,dx=0,$$
since $(u_\lambda^{(s)}-\varepsilon)^+ = 0$ a.e. on $E$, while $a=0$ in $\Omega \setminus E$. 
Therefore $\nabla \big((u_\lambda^{(s)}-\varepsilon)^+ \big) = 0$ a.e. in $\Omega$. This combined with $(u_\lambda^{(s)}-\varepsilon)^+ = 0$ a.e. on $E$ implies that $(u_\lambda^{(s)}-\varepsilon)^+ = 0$ a.e. in $\Omega$. Thus, $$u_\lambda^{(s)}\leq\varepsilon \quad \text{a.e. in } \Omega.$$ Testing likewise with $-(-u_\lambda^{(s)}-\varepsilon)^+$ gives $u_\lambda^{(s)}\geq-\varepsilon$. Consequently,
$$\|u_\lambda^{(s)}\|_{L^\infty(\Omega)}\leq\varepsilon.$$
Since $\varepsilon>0$ was arbitrary, we deduce $u_\infty^{(s)}\equiv0$.

 Finally, one can apply the same argument to every subsequence of $\{u_\lambda^{(s)}\}$, yielding that every convergent subsequence has limit $0$, and therefore
    $$u_\lambda^{(s)}\rightarrow 0 \quad\text{in }C(\overline\Omega)$$
    as $\lambda\to+\infty$, proving the second part of \eqref{eq:clambdau0} and finishing the proof.
\end{proof}

Now that we have characterized both limiting profiles, we focus on studying the non-trivial limiting profile, that is, the one of the ground-state solutions, via studying the variational maximization problem that characterizes it, namely \eqref{eq:varmax}.

\subsection{Properties of the ground-state's limiting profile}

In this section, we study problem \eqref{eq:varmax} under the general
assumptions $(a_*)$, $(a_+)$, and $(g_1)$-$(g_2)$-$(g_3)$. Our aim is to
obtain information about the limiting profile $u_\infty$ by studying
the maximizers of the limiting functional. To this end, define
\begin{equation}\label{eq:defJC}
\mathcal J(u):=\int_\Omega a(x)G(u)\,dx,
\qquad
\mathcal C(u):=\int_\Omega a(x)g(u)\,dx.
\end{equation}
Consider the constrained maximization problem
\begin{equation}\label{eq:P}\tag{P}
\max_{u\in W^{1,\infty}(\Omega)}\mathcal J(u)
\quad\text{subject to}\quad
\mathcal C(u)=0,\quad
\|\nabla u\|_{L^\infty(\Omega)}\leq1.
\end{equation}
Its admissible set is precisely $K_a$. As $u_\infty\in K_a$ maximizes
$\mathcal J$ over $K$, it also solves \eqref{eq:P}. Our first result shows that the two problems have the same
maximizers. In particular, every maximizer over $K$ automatically
satisfies the compatibility condition $\mathcal C(u)=0$.

\begin{lemma}\label{lemma:equivalent-maxima}
The maximum values of $\mathcal J$ over $K$ and over $K_a$ coincide, and so do their sets of maximizers. Moreover, their common maximum value is strictly positive.
\end{lemma}
\begin{proof}
Fix $v\in K$ and consider $h_v(t):=\mathcal J(v+t)$ for $t\in\mathbb R$. By $(g_2)$, $G(s)\sim c|s|^p/p$ as $|s|\to+\infty$. Since $v$ is bounded,
$$\frac{h_v(t)}{|t|^p}\longrightarrow\frac cp\int_\Omega a(x)\,dx<0
\qquad\text{as }|t|\to+\infty.$$
Hence $h_v$ attains its maximum at some $t_v\in\mathbb R$. Since $h_v$ is of class $C^1$,
$$0=h_v'(t_v)=\int_\Omega a(x)g(v+t_v)\,dx.$$
Thus $v+t_v\in K_a$, and $\mathcal J(v+t_v)\geq\mathcal J(v)$. Then, the compactness of $K_a$ (see Lemma~\ref{lemma:compactness}) yields
$$\sup_{v\in K}\mathcal J(v)\leq\max_{v\in K_a}\mathcal J(v).$$
The reverse inequality follows from $K_a\subset K$. In particular, the maximum over $K$ is attained, and every maximizer over $K_a$ is a maximizer over $K$. Conversely, if $u$ maximizes $\mathcal J$ over $K$, then $t=0$ maximizes $h_u$, so $\mathcal C(u)=0$ and $u\in K_a$. Finally, the function $\bar\eta$ from \eqref{def:etabar} is admissible and $\mathcal J(\bar\eta)>0$, proving the positivity of the maximum.
\end{proof}

This equivalence allows us to use perturbations in $K$ without separately preserving $\mathcal C(u)=0$. This is particularly useful when studying the properties of the maximizers of \eqref{eq:P}, since it allows us to use perturbative arguments without worrying about the above mentioned condition. An example of its utility is the following lemma, where we use it to prove that every solution of \eqref{eq:P} must saturate the geometric constraint over the gradient.

\begin{lemma}\label{lemma:saturation}
Any solution $u$ of \eqref{eq:P} satisfies $\|\nabla u\|_{L^\infty(\Omega)}=1$.
\end{lemma}
\begin{proof}
By Lemma~\ref{lemma:equivalent-maxima}, $u$ maximizes $\mathcal J$ over $K$ and $\mathcal J(u)>0$. Let us suppose by contradiction that $\|\nabla u\|_{L^\infty(\Omega)}<1$. For any $\varphi\in C_c^\infty(\Omega)$, one has $u+t\varphi\in K$ for $|t|$ sufficiently small. Hence, by the maximality of $u$ one has
$$0=\left.\frac{d}{dt}\right|_{t=0}\mathcal J(u+t\varphi)
=\int_\Omega a(x)g(u)\varphi\,dx.$$
Since $\varphi$ was arbitrary, 
$$a(x)g(u)=0 \quad \text{a.e. in } \Omega.$$ 
By $(g_1)$, we deduce that $u=0$ a.e. on $\{a\neq0\}$. Since $G(0)=0$, this implies $\mathcal J(u)=0$, contradicting its positivity and finishing the proof.
\end{proof}

Now that we know that the restriction $\|\nabla u\|_{L^\infty(\Omega)}\leq1$ is active for every maximizer of \eqref{eq:P}, let us study regions where this constraint is locally inactive. For $u\in K$, let
\begin{equation}\label{eq:contact-set}
\Lambda_0(u):=\{x\in\Omega:|\nabla u(x)|=1\}
\end{equation}
denote its \emph{contact set}, understood up to a null set. Since $u\in W^{1,\infty}(\Omega)$, its weak gradient is defined only almost everywhere, so a pointwise condition on $\nabla u$ depends on the chosen representative. For this reason, we shall consider
open sets $U\subset\Omega$ uniformly away of $\Lambda_0(u)$, i.e. $\|\nabla u\|_{L^\infty(U)} < 1$. This condition is independent of the representative and ensures
that sufficiently small smooth perturbations supported in $U$
preserve the gradient constraint.

Precisely, using small smooth perturbations in such sets $U$, we were able to prove the following result, describing the behaviour of any maximizer of \eqref{eq:P}.

\begin{proposition}\label{prop:inactive-region}
Let $u\in K_a$ be a maximizer of \eqref{eq:P}, and let $U\subset\Omega$ be an open set such that $\|\nabla u\|_{L^\infty(U)}<1$. Then
$$a(x)g(u)=0\qquad\text{a.e. in }U.$$
In particular, if $a\neq0$ a.e. in $U$, then $u\equiv0$ in $U$. Moreover, if $u\equiv c\neq0$ on an open set $D\subset\Omega$, then $a=0$ a.e. in $D$.
\end{proposition}
\begin{proof}
By Lemma~\ref{lemma:equivalent-maxima}, $u$ maximizes $\mathcal J$ over $K$. Fix $\varphi\in C_c^\infty(U)$, and let $\delta>0$ be such that $\|\nabla u\|_{L^\infty(U)}\leq1-\delta$. For $|t|<\delta/\|\nabla\varphi\|_{L^\infty(\Omega)}$, define $v =u+t \varphi$. See that $|\nabla(u+t\varphi)|\leq1$ a.e. in $U$, while $u+t\varphi=u$ outside $U$. Therefore,  $u+t\varphi\in K$ and the maximality of $u$ yields
$$0=\left.\frac{d}{dt}\right|_{t=0}\mathcal J(u+t\varphi)
=\int_U a(x)g(u)\varphi\,dx.$$
Since $\varphi$ was arbitrary, 
$$a(x)g(u)=0 \quad \text{a.e. in } U.$$ 
If $a\neq0$ a.e. in $U$, as argued before, $(g_1)$ implies $u=0$ a.e. there, and the conclusion follows by continuity. 

Finally, to prove the last statement, it suffices to see that if $u$ were to be equal to a non-trivial constant in an open subset $D  \subset \Omega$, then $\nabla u = 0$ a.e. in $D$, which clearly implies that $\|\nabla u\|_{L^\infty(D)}=0 <1$. Then, by the above argument it follows that
$$a(x) g(u) = 0  \quad \text{a.e. in } D. $$
Since $u \equiv c \not= 0$, $g(u) \not = 0$ and we obtain $a= 0$ a.e. in $D$, concluding the proof.
\end{proof}

The above proposition does not assert that the contact set has full measure wherever $u$ is nonzero, since the absence of an open set with a uniform gap between $|\nabla u|$ and $1$ is weaker than almost-everywhere saturation. Moreover, in regions where $a=0$, the identity $a(x)g(u)=0$ alone gives no information on the value of $u$. To obtain information in such regions for the limiting ground states, we use the following additional property.

\begin{lemma}\label{prop:selection}
Let $u_\infty$ be a limiting profile as in Proposition~\ref{prop:asGS}, and set
$$m:=\max_{v\in K}\mathcal J(v)>0,
\qquad \mathcal S:=\{v\in K:\mathcal J(v)=m\}.$$
Then
\begin{equation}\label{eq:selection}
\Psi(u_\infty)=\min_{v\in\mathcal S}\Psi(v).
\end{equation}
\end{lemma}
\begin{proof}
By Proposition~\ref{prop:asGS}, $u_\infty\in\mathcal S$. Now, fix $v\in\mathcal S \subset K$. Since $u_\lambda$ minimizes $I_\lambda$ over $K$,
$$\Psi(u_\lambda)-\lambda\mathcal J(u_\lambda)
\leq\Psi(v) - \lambda \mathcal J(v) =\Psi(v)-\lambda m.$$
Consequently,
$$\Psi(u_\lambda)\leq\Psi(v)+\lambda\bigl(\mathcal J(u_\lambda)-m\bigr)\leq\Psi(v).$$
Passing to the lim-inf along the subsequence converging to $u_\infty$, and using the lower semicontinuity of $\Psi$, gives 
$$\Psi(u_\infty)\leq\Psi(v).$$ 
Since $v\in\mathcal S$ was arbitrary, the conclusion follows.
\end{proof}

We stress that \eqref{eq:selection} is a necessary property of each limiting ground state; it does not assert that every minimizer of $\Psi$ on $\mathcal S$ is attained as such a limit, nor that this minimizer is unique. On a region where $a=0$, this property provides a simple criterion for the occurrence of constant regions or \emph{plateaux}.

\begin{corollary}\label{cor:limitprofile-plateaus}
Let $u_\infty$ be as in Proposition~\ref{prop:asGS}, and let $D\subset \subset \Omega$ be a connected Lipschitz domain such that $a=0$ a.e. in $D$. Then $u_\infty|_D$ minimizes
$$\int_D\bigl(1-\sqrt{1-|\nabla v|^2}\bigr)\,dx$$
among all $v\in W^{1,\infty}(D)$ satisfying $\|\nabla v\|_{L^\infty(D)}\leq1$ and $v=u_\infty$ on $\partial D$. In particular, if $u_\infty |_{\partial D} = c$, then $u_\infty\equiv c$ in $D$.
\end{corollary}
\begin{proof}
Let $v$ be an admissible competitor on $D$, and define $w=v$ in $D$ and $w=u_\infty$ in $\Omega\setminus D$. Since the traces agree and $D$ is a Lipschitz domain, by standard extension arguments for Sobolev functions one has $w\in W^{1,\infty}(\Omega)$, with $\|\nabla w\|_{L^\infty(\Omega)}\leq1$. Hence $w \in K$ and, since $a=0$ a.e. in $D$,
$$\mathcal J(w)=\mathcal J(u_\infty)=m,$$
so $w\in\mathcal S$. Lemma~\ref{prop:selection} gives $\Psi(u_\infty)\leq\Psi(w)$, which implies that
$$\int_D 1-\sqrt{1-|\nabla u_\infty|^2} \, dx \leq \int_D 1-\sqrt{1-|\nabla v|^2} \, dx,$$ for any admissible $v$, which proves the minimizing property of $u_\infty|_D$.

Finally, if $u_\infty |_{\partial D} = c$, we may take $v\equiv c$ in $D$, obtaining
$$0\leq\int_D\bigl(1-\sqrt{1-|\nabla u_\infty|^2}\bigr)\,dx\leq0.$$
Hence $\nabla u_\infty=0$ a.e. in $D$. Since $D$ is connected and the trace is $c$, $u_\infty\equiv c$ in $D$.
\end{proof}

The results above suggest a picture of the limiting profile in which the weight and the gradient constraint play different roles. Wherever the weight does not vanish, a nonzero profile cannot remain uniformly away from the gradient constraint: on every open set $U$ where $a\neq0$ a.e. and $u_\infty$ is not identically zero, we have $\|\nabla u_\infty\|_{L^\infty(U)}=1$. Regions where the weight vanishes leave more room for the profile to take shape, since changing $u_\infty$ there does not affect $\mathcal J$. The kinetic energy then selects the profile among admissible competitors with the same boundary trace. If the trace is constant on a connected Lipschitz subdomain where $a=0$ a.e., this selection forces a plateau. Nonzero plateaux can therefore occur only where the weight vanishes, while zero plateaux may also occur elsewhere. Moreover, when $g$ is odd, the symmetry $\mathcal J(|u|)=\mathcal J(u)$ and Lemma~\ref{lemma:equivalent-maxima} allow maximizers of \eqref{eq:P} to be chosen nonnegative.

This picture naturally raises a further question: how much of the
limiting profile is in contact with the gradient constraint?
In particular, can $\Lambda_0(u_\infty)$ have full measure, and
what features of the weight favor this behavior? To illustrate
that full saturation does occur, we close with a one-dimensional
example exhibiting a maximizer with a contact set of full measure
(i.e., $|u'|=1$ a.e.) and strictly positive functional value.

\begin{example}\label{ex:final}
Let $\Omega=(-1,1)$, $g(s)=|s|^{p-2}s$ with $p>2$, and
\[
a(x)=
\begin{cases}
-2, & |x|\le \frac12,\\
1, & \frac12<|x|<1.
\end{cases}
\]
In this one-dimensional setting, we can identify the maximizers
of $\mathcal J$ over $K$ and show that their contact sets have
full measure.

As $g$ is odd, we may restrict attention to nonnegative
$u\in K$, because $\mathcal J(u)=\mathcal J(|u|)$.
The sign and geometry of $a$ suggest that an optimal profile
should be low in the middle of the interval and high near its
endpoints, while the Lipschitz constraint determines how quickly
it can pass from one region to the other.

To make this precise, consider the right half of the interval
and fix $b=u(1/2)$. The nonnegativity and $1$-Lipschitz continuity
of $u$ give
\[
u(x)\ge (b+x-\tfrac12)_+
:=\max\{0,b+x-\tfrac12\}
\qquad\text{for }0\le x\le\tfrac12,
\]
and
\[
u(x)\le b+x-\tfrac12
\qquad\text{for }\tfrac12\le x\le1.
\]
Thus the function
\[
v_b(x):=(b+x-\tfrac12)_+
\]
is no larger than $u$ where the weight is negative and no
smaller where the weight is positive. Consequently,
\[
\frac1p\int_0^1 a(x)u(x)^p\,dx
\le \frac1p\int_0^1 a(x)v_b(x)^p\,dx,
\]
with equality only if $u=v_b$ on $[0,1]$. We are therefore
left with choosing the optimal height $b\ge0$. To this end, we maximize the contribution of $v_b$ with respect
to $b$. Set
\[
H(b):=\frac1p\int_0^1 a(x)v_b(x)^p\,dx
=\frac{(b+\frac12)^{p+1}-3b^{p+1}
      +2(b-\frac12)_+^{p+1}}{p(p+1)}.
\]
We claim that $H$ has a unique global maximum at some
$b_p>1/2$. Indeed,
\[
pH'(b)
=(b+\tfrac12)^p-3b^p+2(b-\tfrac12)_+^p.
\]
For $0<b\le1/2$, we have
\[
pH'(b)
=(b+\tfrac12)^p-3b^p
\ge (2^p-3)b^p>0,
\]
because $p>2$, while $pH'(0)=2^{-p}>0$.
For $b>1/2$, dividing by $b^p$ gives
\[
\frac{p}{b^p}H'(b)
=\left(1+\frac1{2b}\right)^p
+2\left(1-\frac1{2b}\right)^p-3.
\]
Hence, setting $r=1/(2b)$, the sign of $H'(b)$ reduces to that of
\[
Q(r):=(1+r)^p+2(1-r)^p-3,\qquad 0<r<1.
\]
This function is strictly convex, with $Q(0)=0$,
$Q'(0)=-p<0$, and $Q(1)=2^p-3>0$.
It therefore crosses zero exactly once in $(0,1)$.
As $r$ decreases when $b$ increases, $H$ first increases and
then decreases, proving the claim.

The same comparison applies on the left half of the interval.
Hence every nonnegative $u\in K$ satisfies
\[
\mathcal J(u)
\le H\bigl(u(-1/2)\bigr)+H\bigl(u(1/2)\bigr)
\le 2H(b_p).
\]
Both halves attain their optimal value at the same height
$b_p$, and the corresponding profiles agree at $x=0$.
Gluing them together gives
\[
u_*(x)=c_p+|x|,
\qquad c_p:=b_p-\frac12>0,
\]
which attains the upper bound and is therefore a global
maximizer. The equality conditions in the preceding comparisons
show that it is the unique nonnegative maximizer.
As $u_*$ is strictly positive, the only maximizers of arbitrary
sign are $u_*$ and $-u_*$.

Finally, the optimality condition $H'(b_p)=0$ reads
\[
(c_p+1)^p-3(c_p+\tfrac12)^p+2c_p^p=0.
\]
Consequently,
\[
\mathcal C(u_*)
=\frac2p\bigl[(c_p+1)^p-3(c_p+\tfrac12)^p+2c_p^p\bigr]=0.
\]
Thus $u_*\in K_a$, and $|u_*'|=1$ almost everywhere in $(-1,1)$.
Moreover,
\[
\mathcal J(u_*)\ge\mathcal J(|x|)
=\frac{2}{p(p+1)}\left(1-3\cdot2^{-(p+1)}\right)>0.
\]
This example provides a strictly positive maximizer with positive
functional value and a contact set of full measure. It also
illustrates how the compatibility condition defining $K_a$ arises
automatically from maximization over $K$.
\end{example}

\appendix

\section{An abstract retraction principle for nonsmooth mountain-pass critical points}
\label{app:mountain-pass}

In this appendix, we recall the mountain pass theorem for
lower semicontinuous functionals due to Szulkin \cite{Sz86}, and
establish an abstract projection principle which will be used in the
proof of the positivity of the mountain pass solution for the particular projection $p(x) = |x|$, inspired by the variational arguments developed in \cite{BCN95}.

\medskip 

Let $(X,\| \cdot \|)$ be a real Banach space and consider a functional
$$
I=\Psi+\Phi:X\rightarrow(-\infty,+\infty],
$$
where $\Psi:X\to(-\infty,+\infty]$ is proper, convex and lower
semicontinuous, while $\Phi\in C^1(X,\mathbb R)$. We recall that
$u\in X$ is a critical point of $I$ if $u\in D(\Psi)$ and
$$
\Phi'(u)(v-u)+\Psi(v)-\Psi(u)\geq 0
\qquad\text{for every }v\in X.
$$
We first recall the version of the mountain pass theorem that is
suitable for the above class of functionals.

\begin{theorem}[{\cite[Theorem 3.2]{Sz86}}]
\label{thm:szulkin-mountain-pass}
Assume that $I=\Psi+\Phi$ satisfies the above hypotheses and the
Palais--Smale condition (cf. Definition \ref{def:PS}). Suppose that
$$
I(0)=0
$$
and that there exist $\rho,\alpha>0$ such that
$$
I(u)\geq\alpha
\qquad\text{for every } \|u\| = \rho,
$$
and that there exists $e\in X\setminus B_\rho(0)$ such that
$$
I(e)\leq 0.
$$
Then $I$ possesses a critical point $u$ at a critical level
$$
c\geq\alpha,
$$
where
$$
c=\inf_{\gamma\in\Gamma}\sup_{t\in[0,1]}I(\gamma(t)),
$$
and
$$
\Gamma
=
\left\{
\gamma\in C([0,1],X):
\gamma(0)=0,\ \gamma(1)=e
\right\}.
$$
\end{theorem}

Its proof relies on the deformation result of Szulkin (cf. \cite[Proposition 2.3]{Sz86}) together with
Ekeland's variational principle. We shall use this deformation argument precisely to extend, for non-smooth functionals, the abstract result presented in \cite[Theorem 10]{BCN95}, which was there used for finding positive solutions for a semilinear problem with an indefinite weight. The main difference is that, since $I$ is not assumed
to be of class $C^1$, we use Szulkin's deformation theorem instead of
a deformation generated by a pseudogradient vector field.

\begin{proposition}
\label{prop:abstract-positivity}
Assume that $I=\Psi+\Phi$ satisfies the hypotheses of
Theorem~\ref{thm:szulkin-mountain-pass}. Let
$$
p:X\rightarrow X
$$
be a continuous map satisfying
$$
p(0)=0,\qquad p(e)=e,\qquad p\circ p=p,
$$
and
\begin{equation}\label{eq:nonexp}
\|p(u)-p(v)\|\leq\|u-v\|
\qquad\text{for every }u,v\in X.
\end{equation}
Assume, moreover, that
$$
I(p(u))\leq I(u)
\qquad\text{for every }u\in X.
$$
Then the mountain pass level
$$
c=\inf_{\gamma\in\Gamma}\sup_{t\in[0,1]}I(\gamma(t))
$$
is attained by a critical point $u\in p(X)$. 
\end{proposition}

\begin{proof}
By Theorem~\ref{thm:szulkin-mountain-pass}, the mountain pass level
$c$ is a critical value of $I$. We shall prove that it can be attained
by a critical point belonging to $p(X)$. Since $p\circ p=p$, we have
$$
p(X)=\operatorname{Fix}(p)
:=\{u\in X:p(u)=u\},
$$
which is closed by the continuity of $p$. Define
$$
\Gamma_p
=
\left\{
\gamma\in C([0,1],p(X)):
\gamma(0)=0,\ \gamma(1)=e
\right\}.
$$
For every $\gamma\in\Gamma$, the path $p\circ\gamma$ belongs to
$\Gamma_p$ and, by hypothesis,
$$
I(p(\gamma(t)))\leq I(\gamma(t))
\qquad\text{for every }t\in[0,1].
$$
Hence
$$
\inf_{\gamma\in\Gamma_p}\sup_{t\in[0,1]}I(\gamma(t))
\leq
\inf_{\gamma\in\Gamma}\sup_{t\in[0,1]}I(\gamma(t))=c.
$$
The reverse inequality follows from $\Gamma_p\subset\Gamma$. Therefore,
\begin{equation}\label{eq:cgammap}
c=
\inf_{\gamma\in\Gamma_p}\sup_{t\in[0,1]}I(\gamma(t)).
\end{equation}

We argue by contradiction and assume that
\begin{equation}\label{eq:HContr}
K_c\cap p(X)=\varnothing,
\end{equation}
where
$$
K_c:=\{u\in X:I(u)=c,\ u\text{ is a critical point of }I\}.
$$
By the (PS)-condition, $K_c$ is compact. Since $p(X)$ is closed and
$K_c\cap p(X)=\varnothing$, we have
\begin{equation}\label{eq:distKc}
d_0:=\operatorname{dist}(K_c,p(X))>0.
\end{equation}
Set
$$
N:=\{u\in X:\operatorname{dist}(u,K_c)<d_0/2\}.
$$
Then $N$ is a neighbourhood of $K_c$ and
\begin{equation}\label{eq:Ndisjoint}
N\cap p(X)=\varnothing.
\end{equation}

Choose $\varepsilon_0\in(0,\min\{c,1\})$. By Szulkin's deformation result
\cite[Proposition 2.3]{Sz86}, applied to the neighbourhood $N$ of
$K_c$, there exists $\varepsilon\in(0,\varepsilon_0)$ with the
following property: every compact set $A\subset X\setminus N$ satisfying
$$
c\leq \sup_{u\in A}I(u)\leq c+\varepsilon
$$
admits the deformation furnished by that proposition. Let
$$
I^{c-\varepsilon/4}:=
\{u\in X:I(u)\leq c-\varepsilon/4\}.
$$
Since $c$ is the mountain pass level, $0$ and $e$ cannot belong to the
same path component of $I^{c-\varepsilon/4}$. Let $W_0$ and $W_e$
denote the path components of $I^{c-\varepsilon/4}$ containing $0$
and $e$, respectively. Define
$$
\Gamma_1^p
=
\left\{
\gamma\in C([0,1],p(X)):
\begin{array}{l}
\gamma(0)\in W_0\cap I^{c-\varepsilon/2},\\
\gamma(1)\in W_e\cap I^{c-\varepsilon/2}
\end{array}
\right\}.
$$
Since $0\in W_0$, $e\in W_e$, and
$$
I(0)=0<c-\varepsilon/2,
\qquad
I(e)\leq0<c-\varepsilon/2,
$$
we have $\Gamma_p\subset\Gamma_1^p$, and in particular
$\Gamma_1^p\neq\varnothing$. Set
$$
c_p:=
\inf_{\gamma\in\Gamma_1^p}
\sup_{t\in[0,1]}I(\gamma(t)).
$$
We claim that
\begin{equation}\label{eq:cpc}
c_p=c.
\end{equation}
Indeed, since $\Gamma_p\subset\Gamma_1^p$, \eqref{eq:cgammap} gives
$c_p\leq c$. Conversely, if some $\gamma\in\Gamma_1^p$ satisfies
$\sup_t I(\gamma(t))<c$, then joining $\gamma(0)$ to $0$ inside
$W_0\subset I^{c-\varepsilon/4}$ and $\gamma(1)$ to $e$ inside
$W_e\subset I^{c-\varepsilon/4}$ would produce a path in $\Gamma$
whose energy is everywhere strictly below $c$, a contradiction.
Thus $c\leq c_p$, proving \eqref{eq:cpc}.

We next verify that $\Gamma_1^p$, endowed with the uniform metric
$$
d(\gamma_1,\gamma_2)
:=\sup_{t\in[0,1]}\|\gamma_1(t)-\gamma_2(t)\|,
$$
is complete. Since $p(X)$ is closed, $C([0,1],p(X))$ is complete.
It therefore suffices to show that $\Gamma_1^p$ is closed in this
space. Let $\gamma_n\in\Gamma_1^p$ and $\gamma_n\to\gamma$ uniformly.
By the lower semicontinuity of $I$,
$$
I(\gamma(0))\leq c-\varepsilon/2,
\qquad
I(\gamma(1))\leq c-\varepsilon/2.
$$
We show that $\gamma(0)\in W_0$; the argument at $t=1$ is identical.
Set $u_n:=\gamma_n(0)$ and $u:=\gamma(0)$. For
$z_{n,t}:=tu_n+(1-t)u$, convexity of $\Psi$ gives
$$
\Psi(z_{n,t})\leq t\Psi(u_n)+(1-t)\Psi(u).
$$
Since $u_n\to u$ and $\Phi$ is continuous,
$$
\sup_{t\in[0,1]}
\big|\Phi(z_{n,t})-t\Phi(u_n)-(1-t)\Phi(u)\big|\rightarrow0.
$$
Consequently,
$$
\sup_{t\in[0,1]}I(z_{n,t})
\leq c-\varepsilon/2+o(1)
\leq c-\varepsilon/4
$$
for $n$ sufficiently large. Hence the segment joining $u_n$ to $u$
is contained in $I^{c-\varepsilon/4}$. Since $u_n\in W_0$, it follows
that $u\in W_0$. Thus $\Gamma_1^p$ is closed and therefore complete.
Moreover, the functional
$$
\Pi(\gamma):=\sup_{t\in[0,1]}I(\gamma(t))
$$
is lower semicontinuous on $\Gamma_1^p$. By \eqref{eq:cpc}, choose $f_0\in\Gamma_1^p$ such that $\Pi(f_0)\leq c+\varepsilon^2$. Applying Ekeland's variational principle
(cf. \cite[Proposition 1.6]{Sz86}) with parameters $\delta=\varepsilon$ and $\lambda=1$, we obtain $f\in\Gamma_1^p$ such that
\begin{equation}\label{eq:pif}
c\leq\Pi(f)\leq c+\varepsilon
\end{equation}
and
\begin{equation}\label{eq:diffPi}
\Pi(g)-\Pi(f)
\geq-\varepsilon d(f,g)
\qquad\text{for every }g\in\Gamma_1^p.
\end{equation}
Set $A:=f([0,1])$. Then $A$ is compact and, by
\eqref{eq:Ndisjoint},
$$
A\subset p(X)\subset X\setminus N.
$$
Furthermore, \eqref{eq:pif} gives
$$
c\leq\sup_{u\in A}I(u)=\Pi(f)\leq c+\varepsilon.
$$
Thus the deformation result of Szulkin applies to $A$. In particular,
there exist $s_0>0$, a neighbourhood $W$ of $A$, and a continuous
family of maps
$$
\alpha_\tau:W\rightarrow X,
\qquad 0\leq\tau\leq s_0,
$$
with $\alpha_0=\operatorname{id}$, such that, for every
$u\in W$ and $0\leq\tau\leq s_0$,
\begin{equation}\label{eq:alpha_sminusu}
\|\alpha_\tau(u)-u\|\leq\tau,
\end{equation}
\begin{equation}\label{eq:IalphaminusIu}
I(\alpha_\tau(u))-I(u)\leq2\tau,
\end{equation}
while
\begin{equation}\label{eq:IalphaminusIu2}
I(\alpha_\tau(u))-I(u)\leq-2\varepsilon\tau
\end{equation}
whenever $I(u)\geq c-\varepsilon$. Moreover,
\begin{equation}\label{eq:IalphaminusIu3}
\sup_{u\in A}I(\alpha_\tau(u))-
\sup_{u\in A}I(u)
\leq-2\varepsilon\tau
\qquad\text{for }0<\tau\leq s_0.
\end{equation}
Note that the deformation $\alpha_\tau$ does not necessarily take values in $p(X)$. To remedy that, we define
$$
\beta_\tau:=p\circ\alpha_\tau.
$$
Since $A\subset p(X)=\operatorname{Fix}(p)$, for every $u\in A$,
\begin{equation}\label{eq:betauminusu}
\|\beta_\tau(u)-u\|
=\|p(\alpha_\tau(u))-p(u)\|
\leq\|\alpha_\tau(u)-u\|
\leq\tau,
\end{equation}
and
\begin{equation}\label{eq:Ibetau}
I(\beta_\tau(u))
=I(p(\alpha_\tau(u)))
\leq I(\alpha_\tau(u)).
\end{equation}
Combining \eqref{eq:Ibetau} with \eqref{eq:IalphaminusIu3}, we obtain
\begin{equation}\label{eq:Ibetausup}
\sup_{u\in A}I(\beta_\tau(u))-
\sup_{u\in A}I(u)
\leq-2\varepsilon\tau.
\end{equation}
Now, we choose $s\in(0,s_0]$ sufficiently small, in particular so that
$2s\leq\varepsilon/2$, and define
$$
f_s:=\beta_s\circ f.
$$
We claim that $f_s\in\Gamma_1^p$. Since $\beta_s$ takes values in
$p(X)$, it remains only to check the endpoints. Consider $t=0$; the
argument at $t=1$ is identical. If $I(f(0))\leq c-\varepsilon$, then,
for every $0\leq\tau\leq s$, \eqref{eq:IalphaminusIu} and
\eqref{eq:Ibetau} give
$$
I(\beta_\tau(f(0)))
\leq I(f(0))+2\tau
\leq c-\varepsilon/2.
$$
If instead $c-\varepsilon<I(f(0))\leq c-\varepsilon/2$, then
\eqref{eq:IalphaminusIu2} and \eqref{eq:Ibetau} yield
$$
I(\beta_\tau(f(0)))
\leq I(f(0))-2\varepsilon\tau
\leq c-\varepsilon/2.
$$
Thus in either case the continuous path
$$
\tau\longmapsto\beta_\tau(f(0)),
\qquad0\leq\tau\leq s,
$$
remains in $I^{c-\varepsilon/2}\subset I^{c-\varepsilon/4}$ and joins
$f(0)$ to $f_s(0)$. Hence
$$
f_s(0)\in W_0\cap I^{c-\varepsilon/2}.
$$
Similarly,
$$
f_s(1)\in W_e\cap I^{c-\varepsilon/2},
$$
so indeed
\begin{equation}\label{eq:fsGamma}
f_s\in\Gamma_1^p.
\end{equation}
Using \eqref{eq:diffPi}, \eqref{eq:fsGamma}, and
\eqref{eq:betauminusu}, we obtain
$$
\Pi(f_s)-\Pi(f)
\geq-\varepsilon d(f,f_s)
\geq-\varepsilon s.
$$
On the other hand, \eqref{eq:Ibetausup} gives
$$
\Pi(f_s)-\Pi(f)
\leq-2\varepsilon s.
$$
This is impossible because $\varepsilon,s>0$. Therefore
$$
K_c\cap p(X)\neq\varnothing.
$$
Hence there exists a critical point $u\in p(X)$ with $I(u)=c$.
\end{proof}

\medskip

\textbf{Data availability:} No data were used for the research described in this article.

\medskip

\textbf{Funding:} The author acknowledges financial support from Universit\`a di Torino through a doctoral scholarship.

\medskip

\textbf{Conflict of interest:} The author declares that he has no conflict of interest.

\end{document}